\documentclass[11pt]{amsart}
\usepackage{amsmath,amsthm,latexsym,amssymb}
\usepackage{enumerate}
\usepackage{graphicx}
\usepackage{color}
\usepackage{epsfig}
\numberwithin{equation}{section}

\newtheorem{theorem}{Theorem}
\newtheorem{claim}{Claim}
\newtheorem{lemma}{Lemma}
\newtheorem{corollary}{Corollary}
\newtheorem{proposition}{Proposition}
\newtheorem{remark}{Remark}
\newtheorem{definition}{Definition}
\newtheorem{quest}{Question}

\numberwithin{theorem}{section}
\numberwithin{corollary}{section}
\numberwithin{lemma}{section}
\numberwithin{definition}{section}
\numberwithin{proposition}{section}
\numberwithin{remark}{section}
\newcommand{\eps}{\varepsilon}

\newcommand{\R}{\mathbb R}
\newcommand{\N}{\mathbb N}

\newcommand{\medint}{-\kern  -,375cm\int}

\def\dist{\mathrm {dist}}
\def\sign{\mathrm {sign}}

\def\boldx{\mbox{\boldmath$x$}}

\def\boldy{\mbox{\boldmath$y$}}

\def\boldmu{\mbox{\boldmath$\mu$}}

\def\boldm{\mbox{\boldmath$m$}}
\def\boldtau{\mbox{\boldmath$\tau$}}

\title[\textsf{The longest shortest fence and sharp Poincar\'e-Sobolev inequalities}]{The longest shortest fence
\vskip.3cm
and sharp Poincar\'e-Sobolev inequalities}
\author[L. Esposito - V. Ferone - B. Kawohl - C. Nitsch - C. Trombetti]{L. Esposito - V. Ferone - B. Kawohl - C. Nitsch - C. Trombetti}


\begin{document}

\begin{abstract}
We prove a long standing conjecture concerning the fencing problem in the plane: among planar 
convex sets of given area, prove that the disc, and only the disc maximizes the length of the shortest
area-bisecting curve. Although it may look intuitive, the result is by no means trivial since we also prove 
 that among  planar convex sets of given area the set which maximizes the length of the shortest bisecting chords
is the so-called \emph{Auerbach triangle}.
\end{abstract}

\maketitle

\section{Introduction}

The aim of this paper is to obtain sharp inequalities involving quantities which are related to the best way of halving 
convex sets in the plane. Problems of this type are called fencing problems, because they model the division of a piece of land into two pieces of equal area. The boundary between the two pieces can then be marked with a fence.
One of the first results in this context is contained in \cite{A}, where the author 
answers some questions posed by M. Ulam in the 1930's concerning the equilibrium position of a cylinder floating in water. 
Subsequently a large amount of literature has been devoted to the study of the length of the arcs 
(straight segments or curves) which divide a convex set in such a way that the area or the perimeter  of the set is bisected 
(see, e.g., \cite{E}, \cite{FP}, \cite{Go}, \cite{Gr}, \cite{GKMS}, \cite{Pa}, \cite{P}, \cite{Ra}, \cite{Rb}, \cite{Sa}).

We will focus our attention mainly on area-bisecting arcs, that is, on curves which split a set into two subsets of equal area.
More precisely, we are motivated by a question posed by P\'olya more than 50 years ago, which has been recently restated, for
example, in \cite[Question 4.3]{Z}. 
\begin{quest}\label{quest1}
In the class of planar convex sets having fixed area, which set maximizes the length of the shortest 
area-bisecting arc?
\end{quest}
P\'olya himself observed in \cite{P} that, if 
$K$ is a convex centrosymmetric set, and if we denote by $|K|$ its area, then an upper bound on the length $L$ of 
the shortest bisecting curve is given by 
\begin{equation}\label{disc}
L\leq 2\sqrt{|K|/\pi}
\end{equation}
i.e., the diameter of a circular disc of area $|K|$. This estimate is sharp only for the disc and provides 
a partial answer to Question \ref{quest1}. Later the same result was found independently by Cianchi \cite{C}.

Note that a set $K$ is centrosymmetric (with respect to its center $0$) if $x\in K$ implies $-x\in K$. 
Following the proof by P\'olya, if $K$ is simply connected and centrosymmetric  and $\bar x\in \partial K$, then 
the chord delimited by $\bar x$ and $-\bar x$ bisects $K$. Inequality \eqref{disc} follows from the fact that
there exists $\bar x\in\partial K$ such that $|\bar x|\le \sqrt{|K|/\pi}$. Indeed, 
either $K$ is a disc or the disc centered at the origin and having same area as $K$ cannot be contained in $K$.
As a matter of fact, it is worth noticing that 
with a slightly different proof, this result holds true in the whole class of measurable centrosymmetric sets, see Proposition \ref{centrosym} at the end of the paper.

Although the restriction to the centrosymmetric case may seem too tight, it suggests that the disc 
provides the answer to Question \ref{quest1}. Until now, however, for general convex domains the best upper 
bound has been obtained as a consequence of Pal's Theorem 
\cite{Pa} (see also \cite[p.37f]{CFG}, \cite[Table 2.1]{SA}, \cite[Ch. 4]{Gr}) and reads 
\begin{equation}\label{sant}
L\leq w(K)\leq  3^{1/4}A^{1/2},
\end{equation}
where $w(K)$ is the width of $K$ (see, e.g. \cite{Sc}). Recall that the width 
of a convex set $K$ is the minimal 
distance between two parallel tangents to $\partial K$. Equality holds in the second inequality of \eqref{sant}  only for an equilateral triangle; therefore the bound on $L$
is not sharp.

The restriction to the centrosymmetric case is also somewhat misleading since one may believe that working with chords instead of curves could be sufficient to answer Question \ref{quest1}. Let us therefore consider the following simpler problem.
\begin{quest}\label{quest2}
In the class of convex sets having fixed area, which set maximizes the length of the shortest bisecting chord?
\end{quest}
This question was posed in \cite[Problem A26]{CFG} where, according to the authors, Santal\'o asked whether the disc is the answer. Actually, Auerbach \cite{A}, by means of an interesting application of Fourier series, already provided a class of set, namely Zindler sets, which give a negative answer to Santal\'o's question. A Zindler set is 
a plane (not necessarily convex) set with the property that all of its (area-)bisecting chords have the same length. Moreover, all area-bisecting chords are also perimeter-bisecting. Clearly the disk belongs to this class, but there are also other domains with this property. In particular, Auerbach focuses on a particular Zindler set, 
now called \emph{Auerbach triangle}, and conjectures that this set has the longest bisecting chord in the class of convex Zindler sets of given area. Auerbach's conjecture has only recently been settled  by Fusco and Pratelli in \cite{FP}.
In view of this result, the Auerbach triangle becomes the natural candidate to answer Question \ref{quest2} 
(see \cite{FP}). Notice that the word \emph{Zindler} is missing between \emph{convex} and \emph{sets} Êin Question 2.

However, it is not difficult to show (see Remark \ref{remark_auerb}) that the shortest area-bisecting arc for the Auerbach triangle is not a straight segment but a circular arc. As a consequence, if the Auerbach triangle provides the answer to Question \ref{quest2}, then one cannot answer Question \ref{quest1} working with chords instead of curves, a fact which makes Question \ref{quest1} even more intriguing.

Indeed, in the present paper we prove
that the disc answers Question \ref{quest1} while the Auerbach triangle answers Question \ref{quest2}.

 In order to give the precise statement of our main results, instead of considering curves splitting $K$ into two subsets 
 of equal area, it is much more convenient to focus attention on the class of subsets of $K$ having area $|K|/2$. 
 More precisely, for any given $E\subset K$, such that $|E|=|K|/2$, the relative perimeter $Per(E;K)$ of $E$ 
 with respect to $K$ represents, in a suitable weak sense, the length of a curve that splits $K$ into two parts of equal area. 
 Taking advantage of this duality between splitting curves and subsets,
 our first main result reads as follows. 

 \begin{theorem}
 \label{main} If $K$ is an open convex set of $\R^2$, we have:
 \begin{equation}\label{mainineq}
 \mathop{\inf_{G \subset K}}_{|G| =|K|/2}Per(G;K)^2\le\frac4\pi|K|.
 \end{equation}
Moreover, equality holds in {\rm(\ref{mainineq})}  if and only if $K$ is a disc.
 \end{theorem}

In other words, the above theorem states that any convex set $K$ that is not a disc must have a bisecting curve strictly shorter than the diameter of the disc of area $|K|$.

The second main result reads as follows.
\begin{theorem}
\label{mainauerb}
Only the Auerbach triangle minimizes area in the class of convex sets whose shortest  area-bisecting chord has given length.
\end{theorem}

In a rephrased form the above theorem states that any convex set $K$, which is not an Auerbach triangle, 
must have an area-bisecting chord which is strictly shorter then the bisecting chord of the Auerbach triangle of the same area $|K|$.

Aside from the fact the fencing problem is interesting in and of itself, Theorem \ref{mainauerb}  is also relevant in connection with geometric dilation (see \cite{DEGKR}),
which is studied in computational geometry, differential geometry and knot theory.
On the other hand,  Theorem \ref{main} has applications in variational problems connected with relative isoperimetric inequalities. We mention here only a few consequences of our Theorem.
  
From now on we consider a bounded open convex set $K \subset \R^2$. It is well known that a relative isoperimetric inequality holds true in the sense that for every $ \alpha \ge1/2$ one can define the relative isoperimetric constant for $K$ as 
 \begin{equation}
 \label{isop_const}
\gamma_\alpha(K) =  \mathop{\inf_{G\not=\emptyset,\> G \subset K}} \dfrac{Per(G;K) }{(\min\{|G|, |K \setminus G|\})^{\alpha}}.
\end{equation}  
As a consequence of Theorem  \ref{main}  we can prove the following corollary.

\begin{corollary}\label{corol_intr}
If $K$ is an open convex set of $\R^2$ and $ \alpha \ge1/2$, we have
\begin{equation}\label{31} 
\gamma_\alpha(K)\le \gamma_\alpha(K^\sharp) ,
\end{equation}
where $K^\sharp$ is the disc of same area as $K$. Moreover, equality holds in {\rm(\ref{31})}  if and 
only if $K$ is a disc. 
\end{corollary}

The isoperimetric inequality comes into play, for example,  in variational problems connected with limiting cases of
nonlinear eigenvalue problems and Poincar\'e type inequalities. For instance, following ideas contained in \cite{FF}, it is shown in \cite{Ga}, \cite{GG} that 
the infimum $\gamma_{1/2}(K)$ coincides with the infimum
 \begin{equation}
 \label{gajew_const}
\Phi(K) =  \mathop{\inf_{u\in BV(K)}}_{u\not=\text{const.}} \dfrac{
\|Du\|(K)
}{\|u-t_0 (u)\|_2},
\end{equation}  
where $\|Du\|(K)$ denotes the total variation of $u$ on $K$ and  the functional $t_0$ is defined by 
 \begin{equation*}
t_0(u) = \sup 
\{t:\> |E_t|\ge|K\backslash E_t|\},\quad E_t= \{x\in\Omega: u(x) > t\}.
\end{equation*}
It has also been demonstrated (see \cite{Ga}) that the infimum $\gamma_{1}(K)$ coincides with the infimum
\begin{equation}
\label{eigen_const}
\mu_1(K) =  \mathop{\inf_{u\in BV(K)}}_{u\not=0,\>\int_K\sign\> u=0} \dfrac{\|Du\|(K)
}{\|u\|_1}.
\end{equation}  
The above problem is the limiting variational formulation for $p=1$ for the first nontrivial eigenvalue relative to the
$p$-Laplacian operator with Neumann boundary conditions.

We finally recall that one can consider a quantity which is somewhat related to the minimum problem \eqref{gajew_const}, that is, the best
constant $ I(K)$ in the following Poincar\'e type inequality 
\begin{equation}
\|Du\|(K)\ge I(K)\|u-\bar u\|_2,\qquad u\in BV(K),
\end{equation}
where $\bar u$ denotes the mean value of $u$ over $K$. It has been proved (see, for example, \cite{C1}) that $I(K)$ can be characterized as
\begin{equation} \label{poinc_const}
I(K)=|K|^{1/2}\mathop{\inf_{G \subset K}}_{0<|G|<|K|} \dfrac{Per(G;K) }{\sqrt{|G|\;|K\setminus G|}}.
\end{equation}
Using Corollary \ref{corol_intr} one can obtain the following result which states some isoperimetric inequalities for the quantities defined above. Incidentally, (\ref{34}) answers an open problem which was recently presented in Oberwolfach \cite{Fuscoetal} (see also \cite[Question 4.1]{Z} and \cite[Problem 4]{BV}).
\begin{corollary}\label{corol_intr1}
If $K$ is an open convex set of $\R^2$, we have: 
\begin{align}\label{32} 
&\Phi(K)\le \Phi(K^\sharp) ,\\ \notag\\
\label{33} 
&\mu_1(K)\le \mu_1(K^\sharp) ,\\ \notag\\
\label{34} 
&I(K)\le I(K^\sharp).
\end{align}
Moreover, equality holds in all of the above inequalities if and 
only if $K$ is a disc. 
\end{corollary}

The paper is organized as follows. In Section \ref{notation} we give some results concerning properties of the shortest bisecting arcs. In particular, we observe that it is possible to relax our minimum problem in \eqref{mainineq} allowing $|G|\le|K|/2$ and we prove that the supremum in problem
\begin{equation}\label{mainine}
\sup_{|K|=\text{const.}} \mathop{\inf_{G \subset K}}_{0<|G| \le\frac{|K|}{2}}Per(G;K)  \end{equation}
is attained, so that we can speak of a maximum. 

In Section \ref{sect_main} we prove Theorem \ref{main}. Our proof consists of three steps. First we prove that any set which attains the maximum 
in \eqref{mainine} has the property that each point of its boundary is a terminal point of a shortest bisecting arc. We name this property \emph{constant halving length}, abbreviated CHL. The other two steps consist of obtaining a parametric representation of CHL-sets and studying
the maximum problem \eqref{mainine} over the class of CHL-sets.

In Section \ref{sect_main10} we prove Theorem \ref{mainauerb}. The proof follows the arguments used in the previous section. The main novelty is an apparently new relaxed formulation which is analogous to \eqref{mainine} (see Proposition \ref{cianchic} below).

Finally, in Section \ref{last} we briefly discuss the centrosymmetric case and we sketch the proof of the corollaries stated above.

\section {Notation and preliminaries}\label{notation}

Let $K$ be an open convex set of $\R^2$. We set
\begin{equation}
\label{const}
\mathcal{C}(K) =  \mathop{\inf_{G \subset K}}_{0<|G| \le \frac{|K|}{2}} \mathcal{Q}(G;K)
\end{equation}
where
\begin{equation}
\label{quotient}
\mathcal{Q}(G;K) =  \dfrac{Per(G;K)^2}{|G|}.
\end{equation}
In what follows we say that a set $E$ is a minimizer for \eqref{quotient} if the minimum in \eqref{const} is attained on $E$.

\begin{proposition}
\label{cianchi}
Let $K$ be an open convex set of $\R^2$. There exists a convex minimizer of \eqref{quotient} whose measure equals $\dfrac{|K|}{2}$, and any minimizer $E$ has the following properties
\begin{enumerate}[(a)]
\item $\partial E  \cap  K$ is either a circular arc or a straight segment. Moreover neither $E$ nor $K \setminus E$ is a circle.
\item Let $P$ be one of the terminal points of $\partial E  \cap  K$. Then $P$ is a regular point of $\partial K$ in the sense that $\partial K$ has a tangent straight line at $P$ and $\partial E  \cap  K$ is orthogonal to $\partial K$.
As a consequence either $E$ or $K \setminus E$ is convex.
\item If $|E|<\dfrac{|K|}{2}$, then $E$ is a circular sector having sides on $\partial K$.
In such a case there exists another minimizer $\hat E$ which is a sector with sides on $\partial K$, having the same vertex as $E$, such that $|\hat E|=\dfrac{|K|}{2}$.
\item If $\partial E  \cap  K$ is a circular arc, its opening angle is at most $\sqrt3$.
\end{enumerate}
\end{proposition}

For the proof of statements \emph{(a)--(c)} we refer to \cite{C}. As regards statement \emph{(d)}, we observe that it is a simple consequence of inequality \eqref{sant}.

We need a slightly more precise information about the terminal points of $\partial E  \cap  K$ where $E$ is a minimizer of \eqref{quotient}.

\begin{proposition}
\label{attainable}
Let $K$ be an open convex set of $\R^2$. If $E$ is a convex minimizer of \eqref{quotient}, then at any terminal point of $\partial E  \cap  K$ the set $K$ satisfies an internal disc condition.
\end{proposition}

\begin{proof}
We give the details of the proof just for the case that  $\partial E  \cap  K$ is not a straight segment.
In view of Proposition \ref{cianchi}\emph{(c)}, the claim is trivial when $|E|<|K|/2$. We fix the origin of our reference frame at the center of the circle, of radius $r$, to which $\partial E  \cap  K$ belongs. Denoting by $P$ one of the terminal points of $\partial E  \cap  K$ we use polar coordinates $(\rho, \theta)$ such that the tangent line to $\partial K$ through $P$ corresponds to $\theta=0$ and the arc $\partial E  \cap  K$ corresponds to $(r,\theta)$ for $0\le\theta\le\theta_0$, where $\theta_0=L/r$ and $L$ is the length of $\partial E  \cap  K$. Clearly, we have $P=(r,0)$ and and we denote by $\theta=\theta(\rho)$ the local parametric representation of $\partial K$ around the point $P$. Arguing by contradiction, we assume that  the set $K$ does not satisfy the internal disc condition at $P$, hence, for every $M>0$, there exists a sequence of points $P_n=(\rho_n,\theta(\rho_n))\in \partial K$, $n\in\N$, such that $P_n$ converges to $P$ and
\begin{equation}\label{tn}
\rho_n\theta(\rho_n)>M(\rho_n-r)^2,\qquad\forall n\in\N.
\end{equation}
Let us consider the sequence of arcs $\{a_n\}_{n\in \N}$ obtained as a perturbation of the optimal arc $\partial E  \cap  K$ which are described in the following way:
\begin{equation}\label{perturb}
r_n(\theta)=r+\eps_n \left(\frac{L}{2}-r\theta\right),
\end{equation}
where
\begin{equation}\label{epsn}
\eps_n=\frac{\rho_n-r}{\displaystyle\frac{L}{2}-r\theta(\rho_n)}.
\end{equation}
We observe that $a_n$ intersects $\partial K$ in $P_n$ and that $\eps_n$ goes to zero as $n$ goes to infinity.

Let us choose $M$ such that
\begin{equation}\label{emme}
M>\frac2L\left(\frac\pi{24} \mathcal{C}(K)+1\right)\ .
\end{equation}
The arc $a_n$ splits the set $K$ in two parts and, if $E_n$ denotes the one with smaller measure, using \eqref{tn},\eqref{perturb},\eqref{epsn} we have
\begin{equation}\label{EEn}
|E|- |E_n|
\le \frac12 \left| \int_0^{\theta_0}r^2\,d\theta-\int_0^{\theta_0}r_n(\theta)^2\,d\theta\right|+o(\eps_n^2)=\frac{L^2\theta_0}{24}\eps_n^2+o(\eps_n^2).
\end{equation}

On the other hand we have
\begin{align*}
Per(E_n;K)-Per(E;K)& \le  \int_{\theta(\rho_n)}^{\theta_0}\sqrt{r_n(\theta)^2+r_n'(\theta)^2}\,d\theta- r\theta_0 \\\notag
&\le \int_0^{\theta_0}\sqrt{r_n(\theta)^2+r_n'(\theta)^2}\,d\theta- \int_0^{\theta(\rho_n)}r_n(\theta)\,d\theta -r\theta_0 \\\notag
&\le \int_0^{\theta_0}\sqrt{r_n(\theta)^2+r_n'(\theta)^2}\,d\theta - \theta(\rho_n)\left(r-|\eps_n|\frac{L}{2}\right)-r\theta_0
\\\notag
&\le \eps_n^2 \frac {L}{2}  - M \left(r-|\eps_n|\frac{L}{2}\right) \frac{(\rho_n-r)^2}{\rho_n} + o(\eps_n^2)\\\notag
&= \eps_n^2 \left(\frac {L}{2} - \frac{M}{\rho_n} \left(r-|\eps_n|\frac{L}{2}\right)\left(\frac{L}{2}-r\theta(\rho_n)\right)^2\right) + o(\eps_n^2)\\\notag
&=  - \frac{L}{2}\left(\eps_n^2 \left(M \frac{L}{2}-1\right)  +  o(\eps_n^2)\right).
\end{align*}
Therefore, for $n$ sufficiently large, in view of \eqref{emme}, we have
$$ \dfrac{Per(E_n;K)^2}{|E_n|}<\dfrac{Per(E;K)^2}{|E|},$$
which contradicts the optimality of $E$.

\end{proof}

\begin{definition}[Optimal arc]
If $E$ is a minimizer of \eqref{quotient} whose measure equals $\dfrac{|K|}{2}$, then we say that $\partial E  \cap  K$ is an optimal arc of $K$.\end{definition}

If $a$ is an optimal arc and $b$ is any other rectifiable simple curve which splits $K$ in two parts of equal measure, then $\ell(a)\le\ell(b)$, where $\ell(\cdot)$ denotes the length of a curve.

An optimal arc has the property to have the shortest possible length among all the halving curves of $E$. For the sake of simplicity, unless otherwise specified, whenever we speak about circular arcs this includes straight line segments (interpreted as circular arcs 
with infinite radius).

\begin{lemma}
\label{succ}
Let $\{K_n\}_{n\in \N}$ be a sequence of open convex sets converging to an open convex set $K$ in the sense of the Hausdorff metric $(\mbox{i.e. }\lim_nd_H(K,K_n)=0)$, then
$ \lim_n \mathcal{C}(K_n) = \mathcal{C}(K)$. In
particular, if $ E_n $ $(n\ge 1)$ is a sequence of convex minimizers of $\mathcal{Q}(\cdot;K_n)$ with $|E_n| = \dfrac{|K_n|}{2}$, every cluster point  of $\{E_n\}_{n\in\N}$ (in the Hausdorff metric) is  a minimizer of $\mathcal{Q}(\cdot;K)$.

\end{lemma}

\begin{proof}
Since $\mathcal{C}(\cdot)$ is invariant under homothety, it is not restrictive to assume that $K_n\subseteq K$.
Let  $E^*$ be any convex set minimizing $\mathcal{Q}(\cdot,K)$. It immediately follows
$$\limsup_n \mathcal{C}(K_n)\!\le\! \limsup_n\dfrac{Per(E^*;K_n)^2}{\min\{|K_n\cap E^*|,|K_n\setminus E^*|\}}\!\le\! \limsup_n\dfrac{Per(E^*;K)^2}{\min\{|K_n\cap E^*|,|K_n\setminus E^*|\}}\!=\!\mathcal{C}(K)$$

On the other hand by Blaschke's selection theorem, see \cite[p. 50]{Sc}, $\{E_n\}_{n\in\N}$  is compact with respect to the Hausdorff metric, and if $E$ is a cluster point of $\{E_n\}_{n\in\N}$ then $|E|=\dfrac{|K|}{2}$ and 
\begin{align}\notag
\mathcal{Q}(E;K)=&\lim_{\varepsilon}\frac{Per(E;K-\varepsilon B)^2}{|E|}\le\lim_{\varepsilon}\liminf_n\frac{Per(E_n;K-\varepsilon B)^2}{|E|}\\ \notag \\ \notag
\le&\liminf_n\frac{Per(E_n;K_n)^2}{|E|}=\liminf_n \mathcal{C}(K_n),\end{align}
where $B$ is the unit disc in $\R^2$ with center at the origin.

\end{proof}

Choosing  $K_n\equiv K$ in Lemma \ref{succ}  we obtain 
\begin{corollary}
\label{compact}
Let ${\mathcal F}(K) $ be the family of area-halving minimizers, that is the family of all minimizers of $\mathcal{Q}(\cdot;K)$  whose measure equals $\dfrac{|K|}{2}$. Then $\mathcal{F}(K)$ is a compact set in the Hausdorff metric.
In particular the  set  ${\mathcal E} ( K)$, subset of $\partial K$, consisting of the all the terminal points of all optimal arcs of $K$ is compact in $\R^2$.
\end{corollary} 

\section {Proof of Theorem \ref{main}}\label{sect_main}

In this section we prove the following result.
\begin{theorem}
\label{main1} For any open convex set $K$ of $\R^2$ we have
\begin{equation}\label{mainineqa}
\mathcal{C}(K) \le \mathcal{C}(K^{\sharp})=\frac8\pi \simeq 2.5464\dots\ .
\end{equation}
Moreover, equality holds in  inequality {\rm (\ref{mainineqa})} if and only if $K$ is a disc.
\end{theorem}

The above result immediately implies Theorem \ref{main}. Indeed, in view of Proposition \ref{cianchi}, we have

$$\mathcal{C}(K) =  \mathop{\inf_{G \subset K}}_{|G| = \frac{|K|}{2}} \mathcal{Q}(G;K)=\frac{2}{|K|}\mathop{\inf_{G \subset K}}_{|G| = \frac{|K|}{2}} Per(G;K)^2,
$$
and this combined with \eqref{mainineqa} implies \eqref{mainineq}.

Theorem \ref{main1} is a consequence of the results contained in the following three subsection, in particular of Propositions \ref{CHLlemma},  \ref{CHL-repr} and  \ref{miniimality}.

\subsection{Reduction to a CHL-set}\label{sub_red}

\begin{lemma}
\label{intersect}
Two optimal arcs of $K$ either cross each other transversally in one and only one point or they coincide.
\end{lemma}
\begin{proof}
Let us consider two optimal arcs $a , a'\subset \bar K$, they certainly satisfy the following properties:
\begin{enumerate}[(p1)]
\item they have the same length $L$;
\item the terminal points belong to $\partial K$, and in those points $\partial K$ has tangent lines;
\item they both split the set $K$ into two subset of equal measure;
\item they are orthogonal to the boundary at each terminal point.
\end{enumerate}
We notice that because of (p3) $a$  and $a'$ have to cross each other in some point inside $K$, and therefore, in view of (p4), unless they coincide, they can not share any terminal point.

We are going to show that if $a $ and $a'$ cross each other twice it is always possible to construct a strictly shorter curve, which splits the set $K$ into two subsets of equal measure.  This contradicts the hypothesis that $a $ and $a'$ are optimal arcs.

Let us assume that $a $ and $a'$ cross each other twice and let $O $ and $O'$
be the centers of the circles to which $a $ and $a'$ belong. Then one of the two cases  certainly occur:  
\vglue .3cm
\noindent Case (i).\ The segment $O O'$ crosses the arcs $a $ and $a'$ (see Figure \ref{intersection}(i)).

\noindent Case (ii).\ The segment $OO'$ does not cross  the arcs $a $ and $a'$ (see Figure \ref{intersection}(ii)).
\vglue .3cm
The case where one of the two arcs degenerates into a straight segment is a limiting case that can be treated as case (i).\medskip
\setlength\fboxsep{30pt}
\setlength\fboxrule{0.0pt}
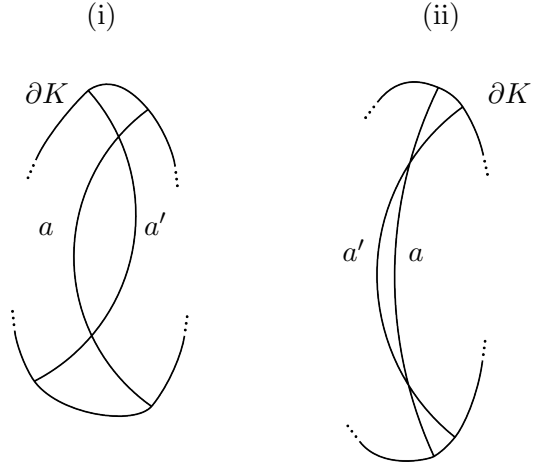
\begin{figure}
\fbox{\includegraphics[width=2.5 in]{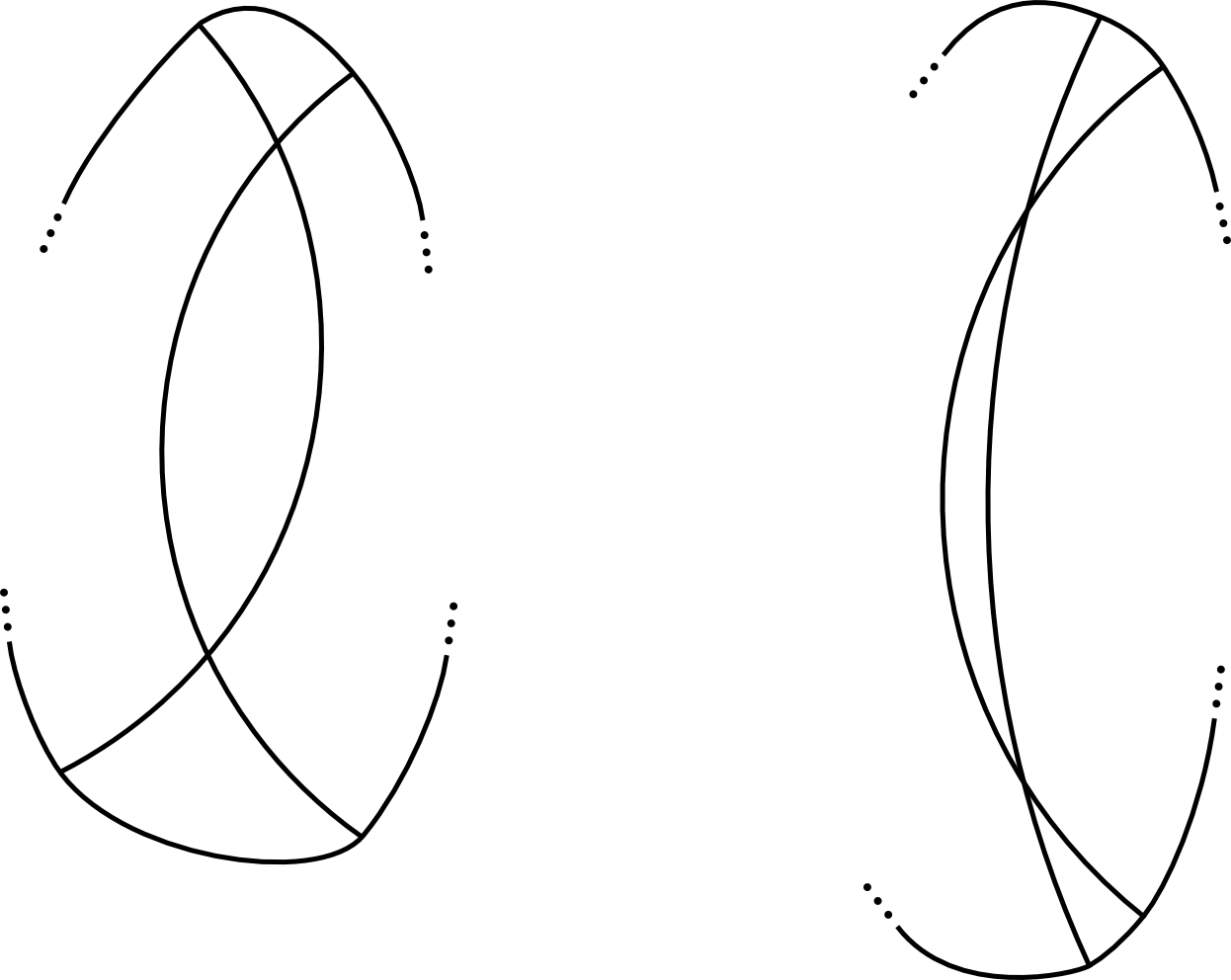}}
\caption{Cases (i) and (ii)}\label{intersection}
\begin{picture}(0,0)(150,-30)
\put(65,170){$\partial K$}
\put(240,170){$\partial K$}
\put(70,120){$a$}
\put(110,120){$a'$}
\put(185,110){$a'$}
\put(210,110){$a$}
\put(88,200){(i)}
\put(215,200){(ii)}
\end{picture}
\end{figure}

\noindent{Case (i)}

The radii 
connecting $O$ with the terminal points of $a$ (see Figure \ref{intersection}(i)) have length $r$ and are tangent to the boundary of $K$. Similarly, the radii 
connecting $O'$ with the terminal points of $a'$ have length $r'$ and are tangent to the boundary of $K$. 
In Figure \ref{case-a}(a) we have drawn the arcs $a$, $a'$ and the points $O$, $O'$. The segment $OO'$ splits $a$ and $a'$ into four arcs $a_1, a_2$ and $a_3, a_4$. 
Then we consider an arc $b_1$ inner parallel to $a_1$ and an arc $b_4$ inner parallel to $a_4$ (see Figure \ref{case-a}(b)), such that they intersect $OO'$ in the same point. Therefore the union of $b_1$ and $b_4$ is an arc of a $C^1$ curve with non empty intersection with $K$ and  it   splits such set into two parts. We will choose $b_1$ and $b_4$ in such a way that $b_1\cup b_4$ splits the set $K$ into two subsets of equal measure. Notice that  it is always possible to satisfy such a condition since, in the limiting case $b_1=a_1$, the arc of curve $a_1\cup b_4$ splits the set $K$ in two parts where the left hand side has a measure smaller than the right one. On the other hand, in the limiting case $a_4 = b_4$,  the  arc of curve $b_1\cup a_4$ splits the set $K$ into two parts where the left hand side has a measure bigger than the right one.
\setlength\fboxsep{30pt}
\setlength\fboxrule{0.0pt}
\begin{figure}
\fbox{
\includegraphics[width=2.5 in]{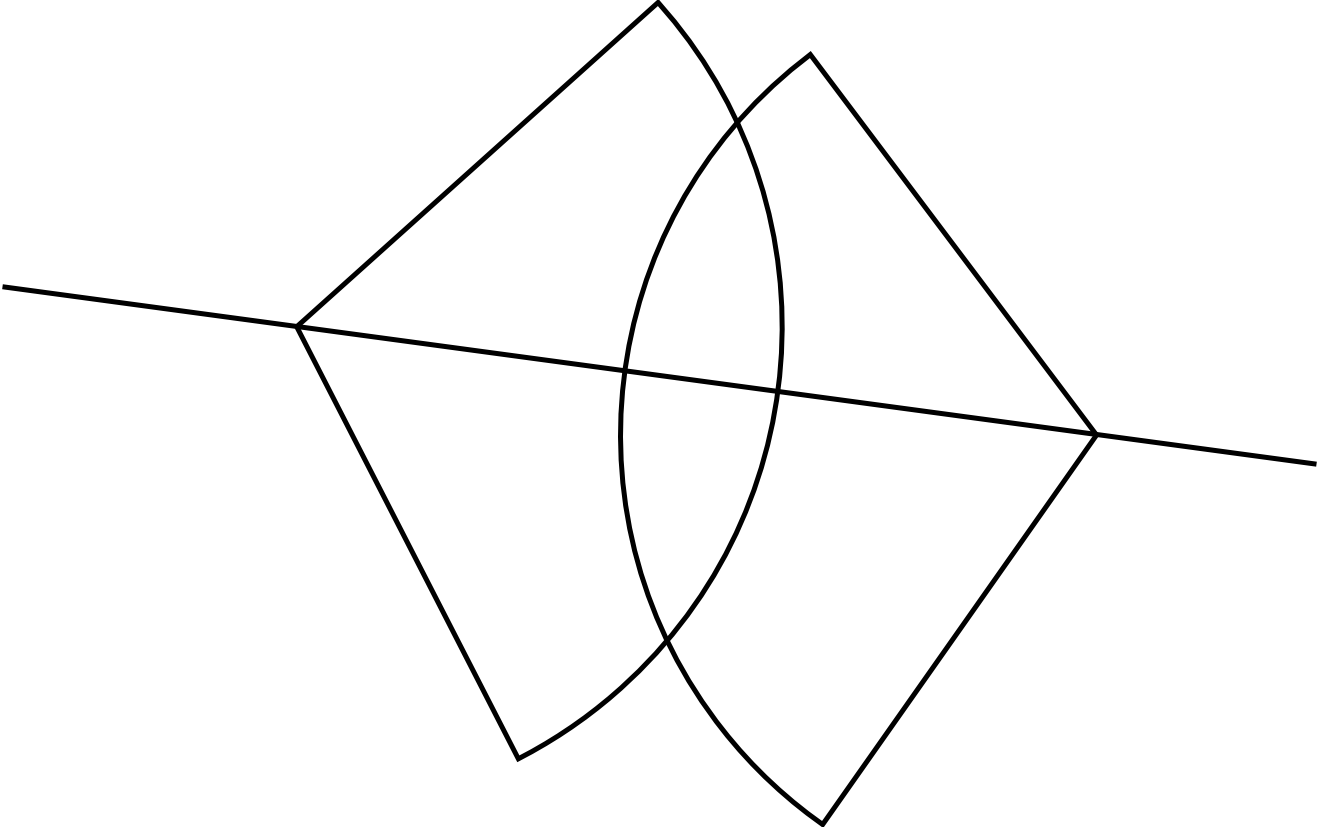}}
\fbox{\includegraphics[width=2.5 in]{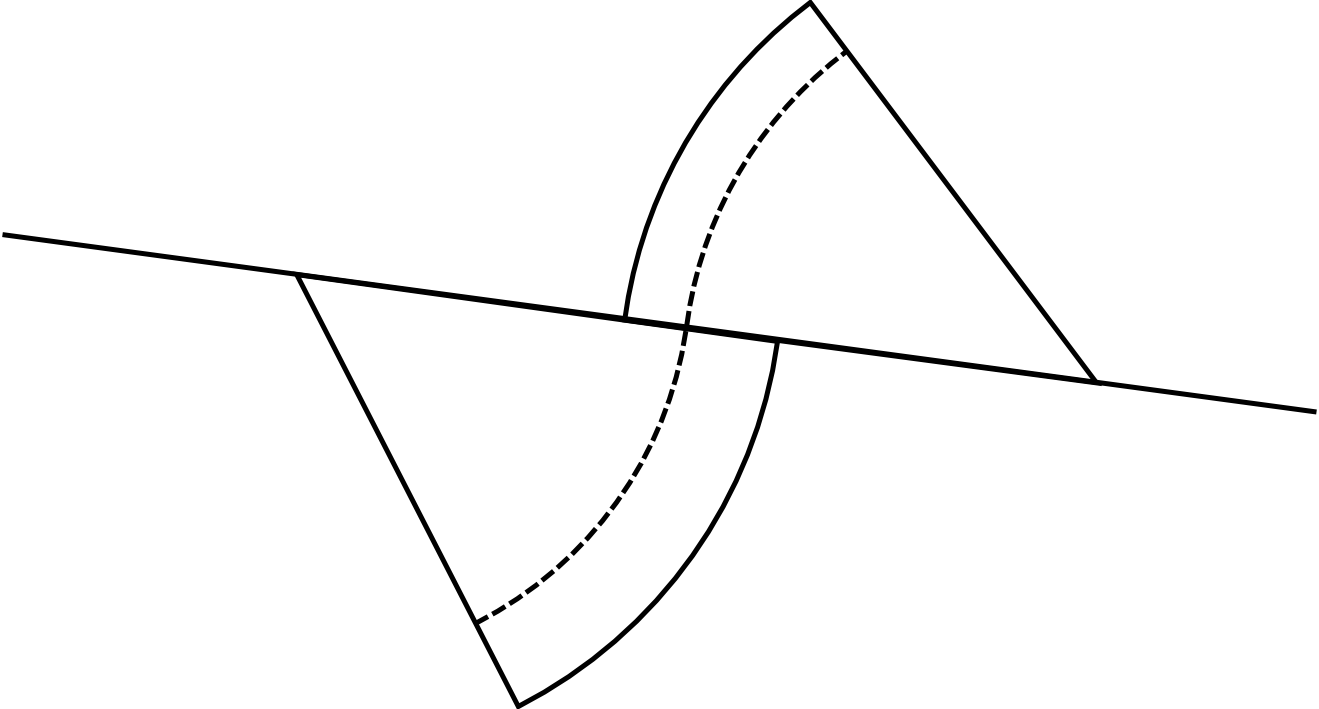}
\includegraphics[width=2.5 in]{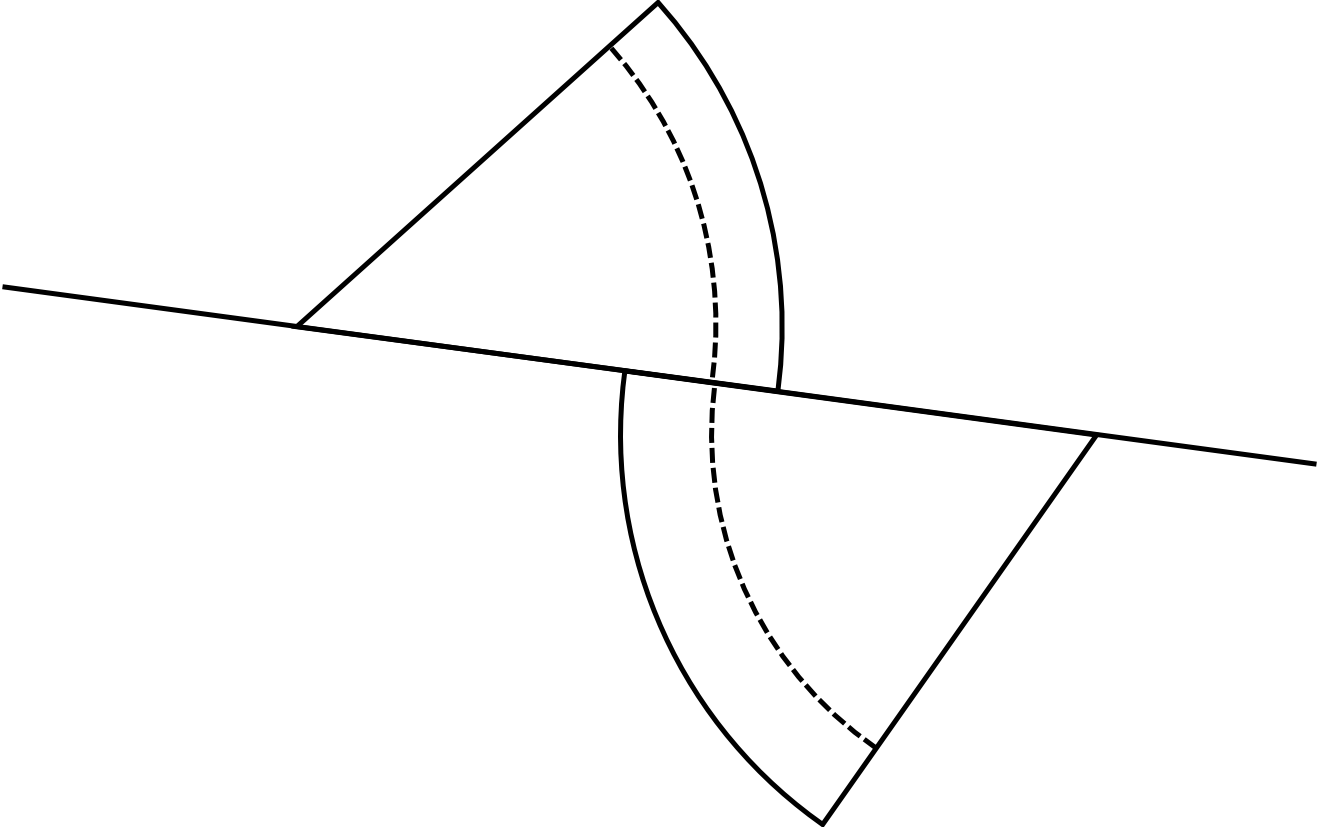}}
\caption{Construction in case (i) }\label{case-a}
\begin{picture}(0,0)(150,-30)
\put(150,335){(a)}
\put(90,285){$O'$}
\put(215,265){$O$}
\put(140,295){$a_1$}
\put(175,285){$a_3$}
\put(130,255){$a_2$}
\put(170,250){$a_4$}
\put(80,160){(b)}
\put(54,120){$a_1$}
\put(77,100){$b_1$}
\put(40,70){$b_4$}
\put(70,55){$a_4$}
\put(260,160){(c)}
\put(228,117){$b_3$}
\put(260,130){$a_3$}
\put(228,58){$a_2$}
\put(260,73){$b_2$}
\put(302,280){$a = a_1 \cup a_2 $}
\put(302,265){$a' = a_3 \cup a_4$}
\end{picture}
\end{figure}
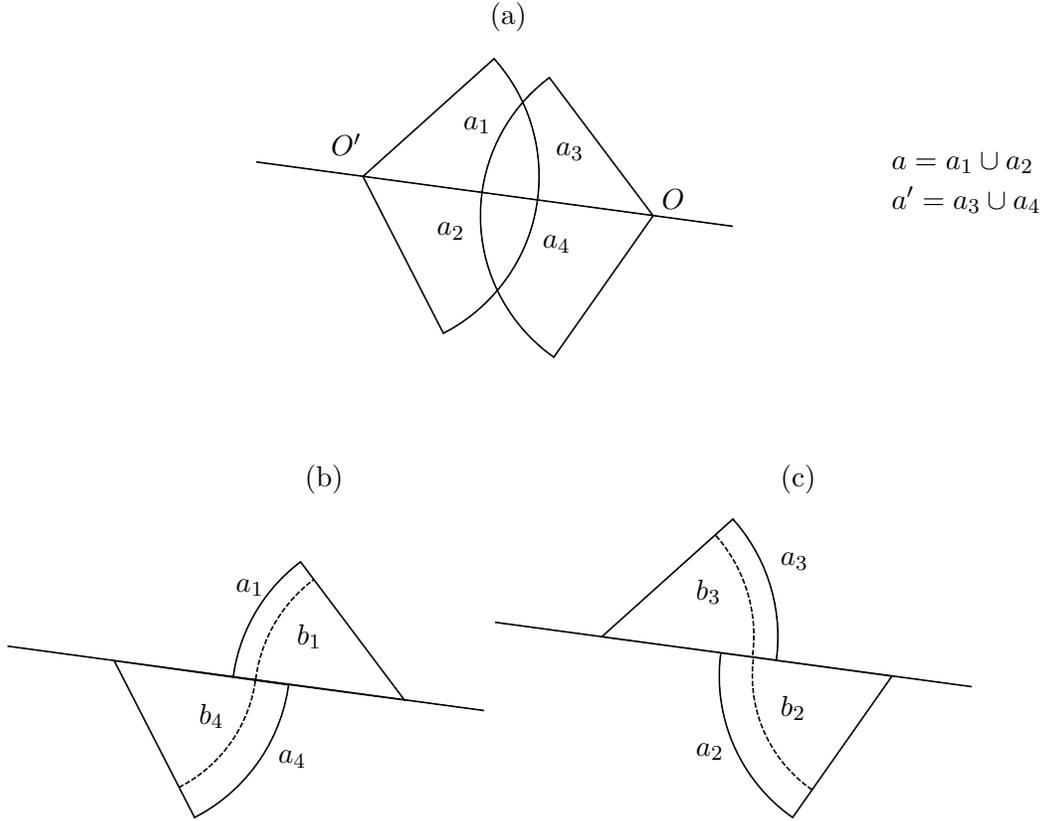

In the same way we can construct an arc $b_3$  inner parallel to $a_3$,  and an arc $b_2$ inner parallel to $a_2$ (see Figure \ref{case-a}(c)), such that they intersect $OO'$ in the same point. Arguing as before the point of intersection can be chosen in such a way that the arc of curve given by the union of $b_3$ and $b_2$ splits the set $K$ into two subset of equal measure.

Since
$$
 \ell( (b_1\cup b_4 )\cap K)+(\ell(b_2 \cup b_3) \cap K)) <\ell(a_1)+\ell(a_2)+\ell(a_3)+\ell(a_4) = 2L
$$
at least one of the two arc of curves  $(b_1 \cup b_4 )\cap K$ or $(b_3 \cup b_2) \cap K$ has length strictly smaller than $L$, which contradicts the hypothesis.\medskip

\noindent Case (ii)
\setlength\fboxsep{30pt}
\setlength\fboxrule{0.0pt}
\begin{figure}
\fbox{{\includegraphics[width=2.5 in]{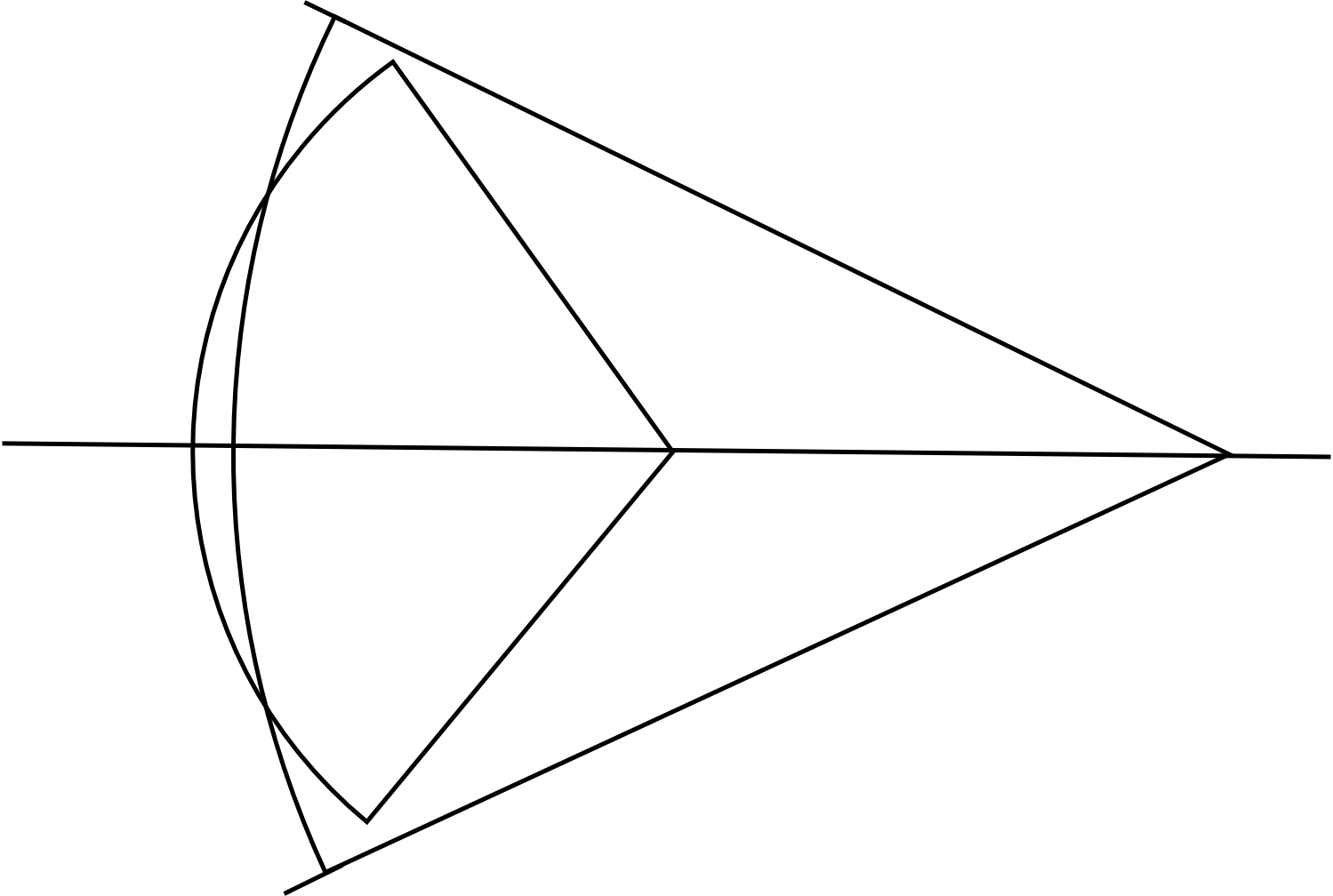}}}
\fbox{{\includegraphics[width=2.5 in, angle=-0.08]{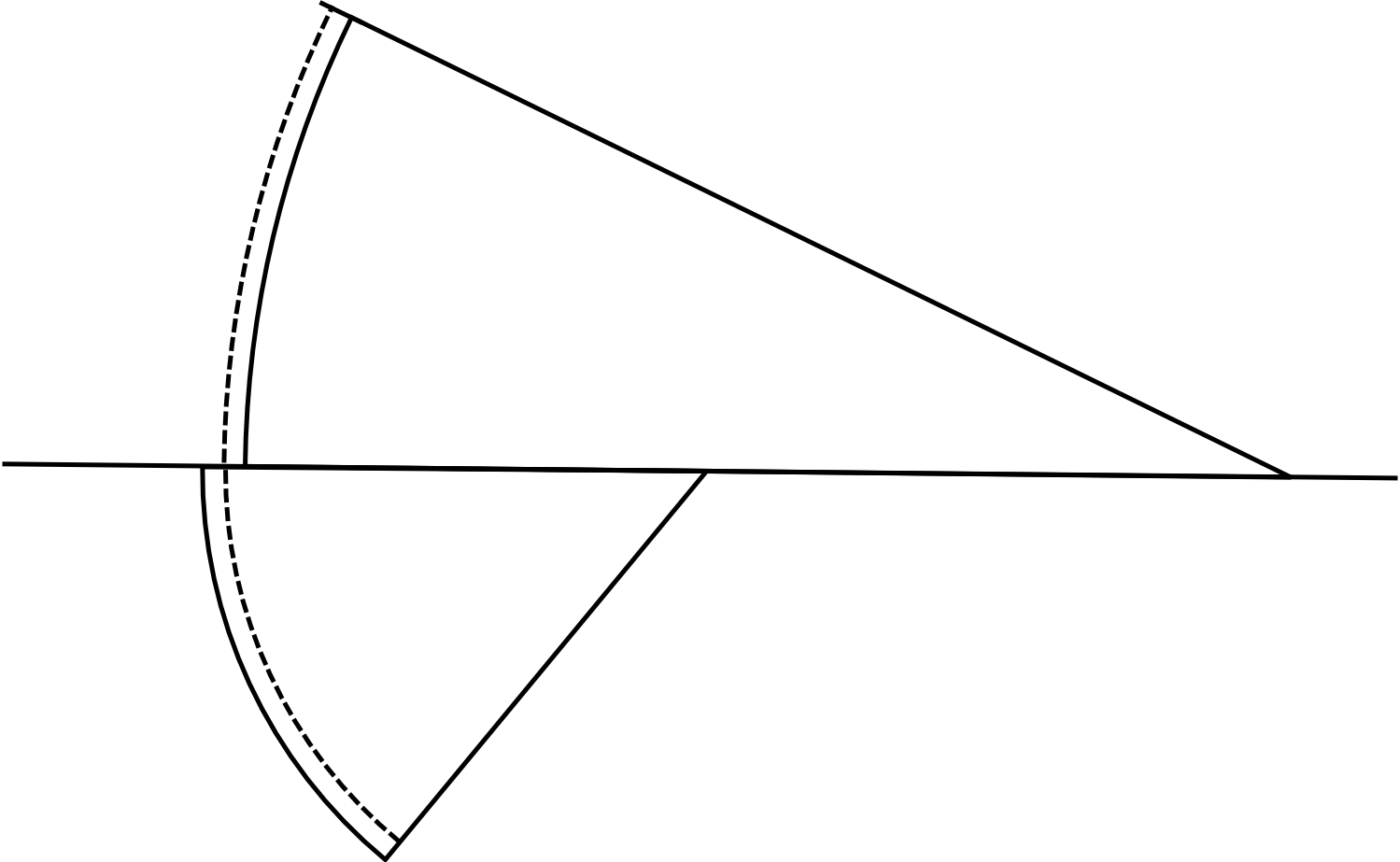}}
{\includegraphics[width=2.5 in,angle=-0.08]{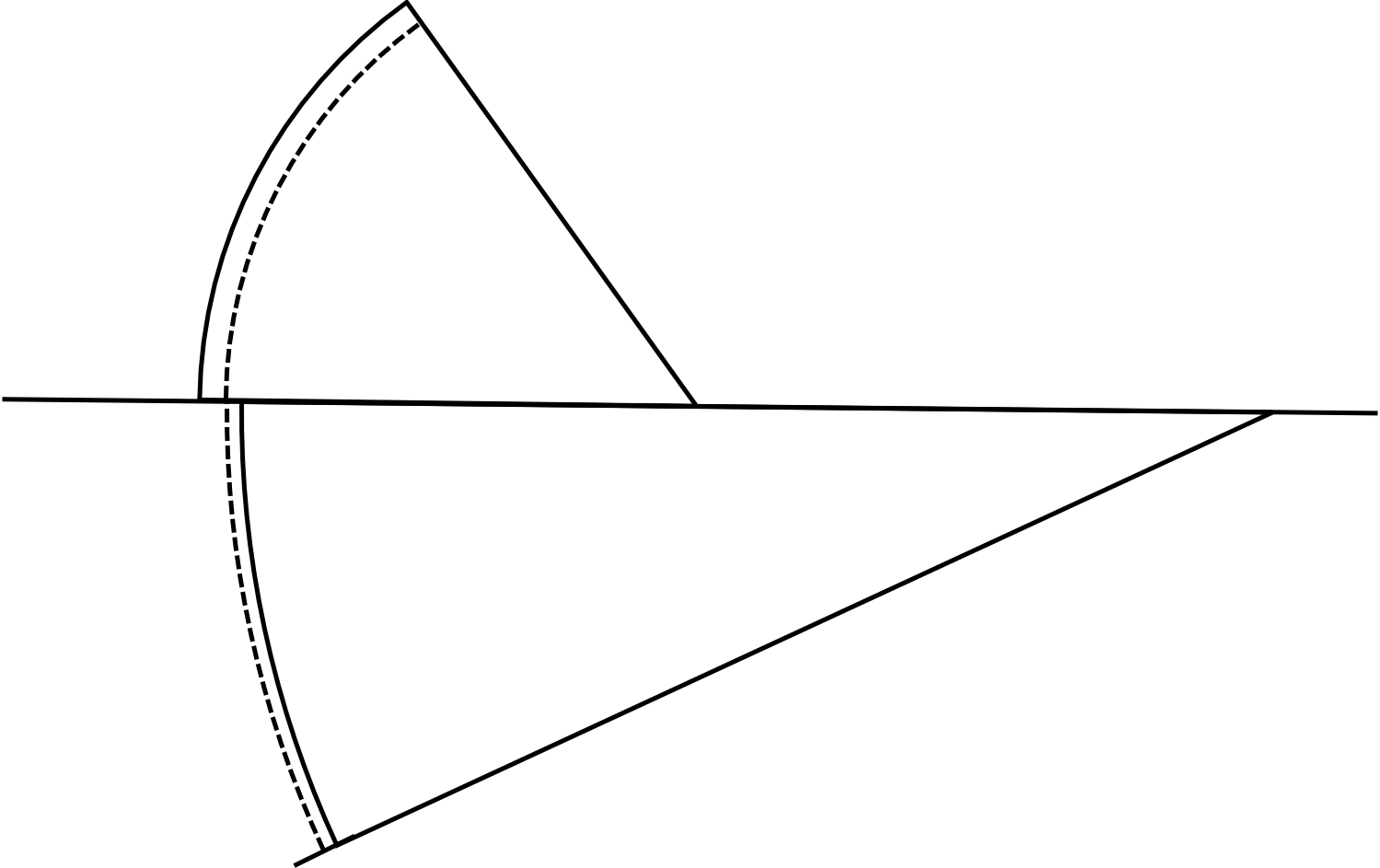}}}
\caption{Construction in case (ii)}\label{case-b}
\begin{picture}(0,0)(150,-30)
\put(150,350){(a)}
\put(150,275){$O'$}
\put(220,275){$O$}
\put(100,280){$a_1$}
\put(72,280){$a_3$}
\put(100,255){$a_2$}
\put(72,255){$a_4$}
\put(40,165){(b)}
\put(-10,120){$b_1$}
\put(7,110){$a_1$}
\put(-15,60){$a_4$}
\put(8,58){$b_4$}
\put(270,165){(c)}
\put(170,120){$a_3$}
\put(190,112){$b_3$}
\put(172,58){$b_2$}
\put(195,60){$a_2$}
\put(300,290){$a = a_1 \cup a_2 $}
\put(300,275){$a' = a_3 \cup a_4$}
\end{picture}
\end{figure}

The radii 
connecting $O$ with the terminal points of $a$ (see Figure \ref{intersection}(ii)) have length $r$ and are tangent to the boundary of $K$. Similarly, the radii 
connecting $O'$ with the terminal points of $a'$ have length $r'$ and are tangent to the boundary of $K$. With the above notation we have that $r'<r$.
In  Figure \ref{case-b}(a) we have drawn the arcs $a$, $a'$ and the points $O$, $O'$. The line passing through $OO'$ splits $a$ and $a'$ in four arcs $a_1, a_2$ and $a_3, a_4$.
Then we consider an arc $b_1$ outer parallel to $a_1$ and an arc $b_4$ inner parallel to $a_4$ (see Figure \ref{case-b}(b)), such that they intersect the line passing through $OO'$ in the same point. Therefore the union of $b_1$ and $b_4$ is an arc of a $C^1$ curve with non empty intersection with $K$ and it  splits such set into two parts. We choose $b_1$ and $b_4$ in such a way that $b_1\cup b_4$ splits the set $K$ into two subset of equal measure. 

In the same way we can construct an arc $b_3$ inner parallel to $a_3$,  and an arc $b_2$ outer parallel to $a_2$ (see Figure \ref{case-b}(c)), such that they intersect the line passing through $OO'$ in the same point and the union of $b_3$ and $b_2$ splits the set $K$ into two subset of equal measure.

Now we want to prove that at least one of the two arcs of curves  $(b_1 \cup b_4 )\cap K$ or $(b_3 \cup b_2) \cap K$ has length strictly smaller than $L$, which contradicts the hypothesis.

Given a circular arc $c$ we denote by  ${\mathcal S}(c)$  the circular sector delimited by   $c$  and  by $\theta(c)$  its opening angle. For every $i=1,\dots,4$ we define $\varphi_i \ge 0$ such that
\begin{equation*}
\theta(a_i) = \theta(b_i \cap K) + \varphi_i.
\end{equation*}
If $u_i= {\rm dist} (a_i,b_i)$, $i=1,\dots,4$, we have:
\begin{eqnarray}
\label{deltap}
\ell(b_i \cap K) = \ell(a_i) + \theta(a_i) u_i -  (r+u_i) \varphi_i,
 \quad i=1,2 \\
\cr
\label{deltapp}  \ell(b_i \cap K) =  \ell(a_i) - \theta(a_i) u_i - (r'-u_i)\varphi_i, \quad i=3,4.
\end{eqnarray}

On the other hand, we observe that
$$
\sum_{i=1}^{4} |{\mathcal S}(b_i ) \cap K| =\sum_{i=1}^{4} |{\mathcal S}(a_i) \cap K| = 
 |K|.
$$
The terms with $i=1$ in the above sums are such that \begin{align}
\label{esse}|{\mathcal S}(b_1 ) \cap K|-|{\mathcal S}(a_1) \cap K|&=-|{\mathcal S}(a_1) | + |{\mathcal S}(b_1) |-|({\mathcal S}(a_1) \bigtriangleup   {\mathcal S}(b_1) ) \setminus K  |\\ \notag\\
\notag &=\theta(a_1) u_1(r+u_1/2)-|({\mathcal S}(a_1) \bigtriangleup   {\mathcal S}(b_1) ) \setminus K  |.
\end{align}
A similar computation holds true for every $i=1,\dots,4$, and, using  \eqref{deltap} and \eqref{deltapp}, we get:
\begin{align}
\label{deltaa}
0&=\displaystyle  
\sum_{i=1}^{4} |{\mathcal S}(b_i ) \cap K|-\sum_{i=1}^{4} |{\mathcal S}(a_i) \cap K| \\
\notag
&=\sum_{i=1}^{4} |{\mathcal S}(b_i) |
-\sum_{i=1}^{4} |{\mathcal S}(a_i) |
-\sum_{i=1}^2   |({\mathcal S}(a_i) \bigtriangleup   {\mathcal S}(b_i) ) \setminus K  | +\sum_{i=3}^4   |({\mathcal S}(a_i) \bigtriangleup   {\mathcal S}(b_i) ) \setminus K  | \\
\notag
&=\theta(a_1) u_1(r+u_1/2)+\theta(a_2) u_2(r+u_2/2)- \theta(a_3) u_3(r'-u_3/2) - \theta(a_4) u_4(r'-u_4/2)\\
\notag
& \displaystyle-\sum_{i=1}^2   |({\mathcal S}(a_i) \bigtriangleup   {\mathcal S}(b_i) ) \setminus K  | +\sum_{i=3}^4   |({\mathcal S}(a_i) \bigtriangleup   {\mathcal S}(b_i) ) \setminus K  |\\
\notag
& \displaystyle=r' \left ( \sum_{i=1}^{i=4}  \ell(b_i \cap K) - \sum_{i=1}^{i=4}  \ell(a_i) \right)+(r-r')\theta(a_1)u_1+(r-r')\theta(a_2)u_2+
 \sum_{i=1}^{i=4} \theta(a_i) u_i^2/2\\
\notag
& \displaystyle+r'(r+u_1)\varphi_1+r'(r+u_2)\varphi_2+r'(r'-u_3)\varphi_3+r'(r'-u_4)\varphi_4\\
\notag
& \displaystyle
-\sum_{i=1}^2   |({\mathcal S}(a_i) \bigtriangleup   {\mathcal S}(b_i) ) \setminus K  | +\sum_{i=3}^4   |({\mathcal S}(a_i) \bigtriangleup   {\mathcal S}(b_i) ) \setminus K  |
\end{align}

\medskip

Observe that 
\begin{eqnarray}
\label{maggioo}
 |({\mathcal S}(a_i) \bigtriangleup   {\mathcal S}(b_i) ) \setminus K  | 
 \leq (r u_i+u_i^2/2) \varphi_i, \quad i= 1,2.
\end{eqnarray}
\setlength\fboxsep{30pt}
\setlength\fboxrule{0.0pt}
\begin{figure}
\fbox{{\includegraphics[scale=0.75, width=3. in]{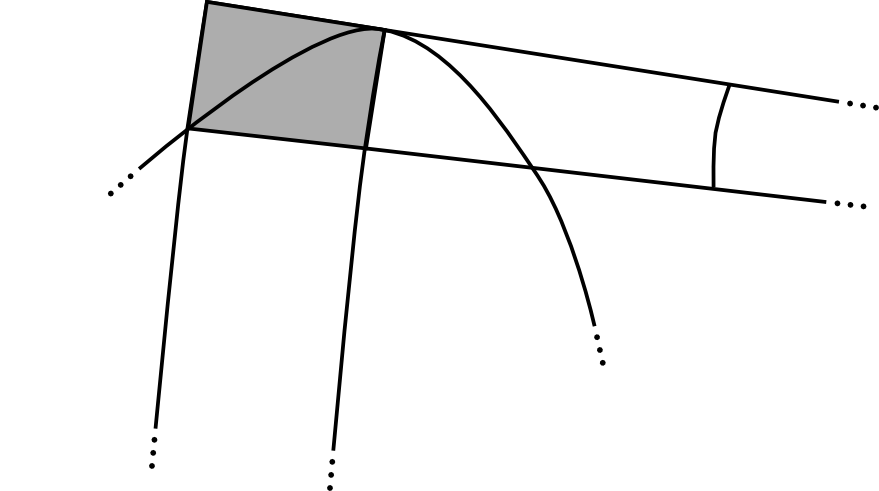}}}
\caption{Visualization of (\ref{maggioo}) }\label{deltaarea}
\begin{picture}(0,0)(150,-30)
\put(200,120){$\varphi_1$}
\put(188,82){$\partial K$}
\put(70,85){$b_1$}
\put(115,85){$a_1$}
\put(82,156){$B$}
\put(135,150){$C$}
\put(135,120){$D$}
\put(77,125){$A$}
\end{picture}
\end{figure}

Figure \ref{deltaarea} refers to the case $i=1$, and the inequality \eqref{maggioo} follows from the fact that 
the measure of $ABC$ is smaller than the measure of $ABCD$. 

From \eqref{deltap}, \eqref{deltapp}, \eqref{deltaa} and \eqref{maggioo} we have
\begin{equation*}
\displaystyle r' \left ( \sum_{i=1}^{i=4}  \ell(b_i \cap K) - \sum_{i=1}^{i=4}  \ell(a_i) \right) <
\varphi_1[u_1(r+u_1/2) -r' (r+u_1)] + \varphi_2[u_2(r+u_2/2) -r' (r+u_2)] <0.
\end{equation*}
Hence
$$
\ell((b_1 \cup b_4) \cap K) + \ell((b_2 \cup b_3) \cap K) <2L
$$
and the claim follows.

\end{proof}

\begin{corollary}\label{coroll}
If $E_1$ and $E_2$ are two circular sectors having the sides on $\partial K$, $E_1$ and $E_2$ minimize \eqref{quotient}, $|E_1|=|E_2|=\frac{|K|}{2}$, then the two sectors necessarily share one of the two sides. \end{corollary}
\begin{proof}
The proof is a straightforward consequence of the previous lemma since the optimal arcs enclosing the circular sectors $E_1$ and $E_2$ have to cross each other in one and only one point.

\end{proof}

We are now ready to characterize the set maximizing $\mathcal{C}(\cdot)$. First we need the following compactness result.

\begin{lemma}
\label{max}
Let $C_0 = \sup_{K} \mathcal{C}(K)$, then there exists a convex set $K_0$ such that 
$\mathcal{C}(K_0) = C_0$
\end{lemma}

\begin{proof}
Let us consider a maximizing sequence of convex sets $\{K_n\}_{n\in\N}$. Since $\mathcal{C}(\cdot)$ is invariant
under homothety, we can suppose that $|K_n|=1$ for $n\in\N$. We claim that the diameter $d(K_n)$ of $K_n$
can be uniformly bounded. Indeed, see \cite{SA},
\begin{equation*}
d(K_n)\le 2\frac{|K_n|}{w(K_n)},
\end{equation*}
where $w(K)$ denotes the width of a convex set $K$, that is the minimum distance of two nonidentical parallel tangent planes to $K$. Furthermore, by definition of width and $\mathcal{C}(\cdot)$,  we have
\begin{equation*}
2 w(K_n)^2\ge  \mathcal{C}(K_n){|K_n|}
\end{equation*}
and the claim is proven.

The compactness of $\{K_n\}_{n\in \N}$ in $\mathcal K$ (the class of planar convex sets endowed with the Hausdorff metric) is provided by Blaschke selection theorem, see \cite[p. 50]{Sc}, and the proof of the Lemma follows from the continuity of $\mathcal{C}(\cdot)$ in $\mathcal K$ as stated in Lemma \ref{succ}.

\end{proof}

\begin{definition}[CHL - set]\label{CHLdef}
We say that a convex set $K$ is a  set with constant halving length (CHL - set ) if 
 each point of its boundary is a terminal point of an optimal arc.
\end{definition}

\begin{proposition}\label{CHLlemma}
If $K^*$ is a convex set with $\mathcal{C}(K^*) =\sup_{K} \mathcal{C}(K)$, then $K^*$ is a CHL - set.
\end{proposition}

\begin{proof}
Let $K^*$ be a convex set with $\mathcal{C}(K^*) =\sup_{K} \mathcal{C}(K)$. The strategy of the proof consists in showing that, if $K^*$ is not a CHL-set, then it is always possible to ``cut off a piece of $K^*$'' in such a way that $\mathcal{C}(\cdot)$ strictly increases.
We argue by contradiction supposing  that $K^*$ is not a CHL-set, that is, $\partial K^*\setminus \mathcal E(K^*)$ is nonempty.

\begin{claim}\label{claim1}
It is possible to find an open half-plane $H$ such that $H\cap K^*\not=\emptyset$ and $H\cap \partial K^*$ does not contain points of $\mathcal E(K^*)$.
\end{claim}
Because of Corollary \ref{compact}, it is possible to find an open connected arc $a$ on $\partial K^*$, not containing points of $\mathcal E(K^*)$, such that its ending points $P_1, P_2$ belong to  $\mathcal E(K^*)$.
If we consider the line $l$ passing through $P_1$   and $P_2$, two cases may occur.
\begin{enumerate}[(i)]
\item The set $l\cap \partial K^*$ is not a segment.
\item The set $l\cap \partial K^*$ is a segment.
\end{enumerate}

In case (i) the Claim \ref{claim1} is proved choosing $H$ as the open half-plane bounded by $l$ and containing $a$.

In case (ii) we consider the points $P'_1$   and $P'_2$ which are the second terminal points of the optimal arcs passing through $P_1$   and $P_2$ and we denote by $a'$ the arc of $\partial K^*$ with end points $P'_1$   and $P'_2$ (not containing $P_1$   and $P_2$). Clearly, if $l'$ is the line passing through $P'_1$   and $P'_2$, the set $l'\cap \partial K^*$ cannot be a segment since two intersecting circular arcs cannot be both orthogonal to two different lines.
Claim \ref{claim1} is then proved choosing $H$ as the half-plane bounded by $l'$ and containing $a'$. Indeed,  the arc $a'$ does not contain points of $\mathcal E(K^*)$. This comes from the fact that any optimal arc $b$ having a terminal on $a'$ cannot have the second terminal on $a$. Hence $b$ has to intersect the optimal arc connecting $P_1$ to $P'_1$ (or $P_2$ to $P'_2$) in two points, but this is not possible in view of Lemma \ref{intersect}.

\begin{claim}\label{claim2}
It is possible to find an open half-plane $H'$ contained in $H$ such that $H'\cap K^*\not=\emptyset$ and $\mathcal{C}(K^* \setminus {\overline {H'}} )\ge \mathcal{C}(K^*)$.
\end{claim}

We denote by $\{H_t\}_{t\ge 0}$,  the one-parameter family of open half-planes  such that $\dist(H_t,\partial H) =t$. Obviously, $H_0 = H$,  $H_{t_1} \subset H_{t_2}$ if $t_1 > t_2$.

Lemma \ref{succ} and Corollary \ref{compact} ensure that  there exists ${\bar t}$ such that $H_{\bar t}\cap K^*\not=\emptyset$ and $\mathcal E(K^* \setminus {\overline {H_{\bar t}}})\subset\partial K^* $.
Choosing $H' = H_{\bar t}$ we observe that any optimal cut of $\tilde K=K^* \setminus {\overline {H'}}$  touches $\partial K^*$ and bounds a portion of $K^*$ with measure strictly smaller than $|K^*|/2$. Therefore there exists a set $E\in\mathcal F(\tilde K)$ such that $Per(E;\tilde K) =Per(E;K^*)$. Using the definition of $\mathcal{C}(\cdot)$ and Proposition \ref{cianchi}, we get
\begin{equation}
\label{CK}
\mathcal{C}(\tilde K)=  \dfrac{Per(E;\tilde K)^2}{|E|}=\dfrac{Per(E; K^*)^2}{|E|}\ge \mathop{\inf_{G \subset K^*}}_{0<|G| \le \frac{|K^*|}{2}} \mathcal{Q}(G;K^*)=\mathcal{C}(K^*).
\end{equation}
Claim \ref{claim2} is proved. 

\medskip

By Claim \ref{claim2}, if
$$\mathcal{C}(K^* \setminus {\overline {H'}} )> \mathcal{C}(K^*),$$
we get a contradiction and the proposition is proved.

On the other hand, it may happen that
$$\mathcal{C}(K^* \setminus {\overline {H'}} )= \mathcal{C}(K^*).$$
This means that \eqref{CK} holds as a chain of equalities. But then, in view of Proposition \ref{cianchi}\emph{(c)}, the set $E$ is a circular sector with sides on $\partial K^*$.
According to Corollary \ref{coroll} there exist up to three optimal circular sectors in $\mathcal F(K^*)$ with sides on $\partial K^*$. We also observe that any optimal circular sectors in $\mathcal F(K^*)$ have the same opening angle $\alpha=\mathcal{C}(K^*)/2$. 

Three cases may occur.
\begin{enumerate}[(a)]
\item There exists only one optimal circular sector in $\mathcal F(K^*)$.
\item There exist only two optimal circular sectors in $\mathcal F(K^*)$.
\item There exist  three optimal circular sectors in $\mathcal F(K^*)$.
\end{enumerate}

In case (a) we proceed as in Claim \ref{claim2}. We cut a suitably small piece of $K^*$ close to the vertex of the circular sector by a line which is orthogonal to the bisector of the sector,
and it is easy to prove that \eqref{CK} holds as a strict inequality. The contradiction follows.

Similarly, in case (b) we cut off two suitably small pieces of $K^*$, with equal measure, close to the two vertices of the circular sectors by two lines which are orthogonal to the bisectors of each sector,
and it is easy again to realize that \eqref{CK} holds as a strict inequality. The contradiction follows.

In case (c) we observe that $K^*$ has to be an equilateral triangle and the contradiction immediately follows (see \cite{C}), since $\mathcal{C}(K^*)$ would be equal to $2\pi/3$. But $2\pi/3<8/\pi$, the value  of $\mathcal{C}(\cdot)$ on a disc.

\end{proof}

\subsection{Regularity and representation of CHL-sets}

Let $K$ be a CHL-set. As a consequence of Proposition \ref{attainable} we obtain the following regularity result for $K$.

\begin{lemma}\label{disk_cond}
Any CHL-set  $K$ is of class $C^{1,1}$.
\end{lemma}
\begin{proof}
In view of Proposition \ref{attainable} any CHL-set $K$ is of class  $C^{1}$. We have to prove that it is of class $C^{1,1}$. For every $P\in\partial K$, we denote by $D_{R}(P)$ the open disc  of radius $R$, tangent to $\partial K$ in $P$, having nonempty intersection with $K$. Using the arguments of the proof of 
Proposition \ref{attainable}, if $L$ is the length of the optimal arcs of $K$, we deduce that there exists a neighborhood $\mathcal I$ of $P$ such that for every
\begin{equation*}
\bar R<\frac1{\frac4L\left(\frac\pi{24} \mathcal{C}(K)+1\right)}
\end{equation*}
we have  $D_{\bar R}(P)\cap\mathcal I\subset K$. 
Observing that the choice of $\bar R$ is independent of $P$, the lemma follows.

\end{proof}

\begin{remark}
We observe that, employing   \eqref{sant}, any CHL-set having optimal arcs of length $L$, satisfies a uniform internal disc condition depending only on $L$. One can choose, for instance,
the radius of the disc
\begin{equation}
\label{unifdisc}
\tilde R=\frac{L}{ 8\left(\frac{\pi\sqrt{3}}{12} +1\right)}.
\end{equation} 

\end{remark}

For any given $\boldx\in\partial K$ there exists a unique $\boldy\in\partial K$ such that $\boldx$ and $\boldy$ are terminal points of an optimal arc $\Gamma$. In view of Lemma \ref{succ} such correspondence is a continuous bijection from $\partial K$ into itself. Moreover, because of Lemma \ref{intersect}, $\boldmu(\boldx) = \dfrac{\boldy-\boldx}{|\boldy-\boldx|}$ is a continuous invertible map from $\partial K$ into $\mathcal{S}^1$. We observe that, because of Proposition \ref{cianchi}\emph{(d)}, $|\boldy-\boldx|$ is bounded away from 0. Therefore, for any given $\sigma\in[-\pi,\pi)$, we denote by $\boldx(\sigma)$ and $\boldy(\sigma)$ the two terminal points of the unique optimal arc $\Gamma_\sigma$ such that
$$\dfrac{\boldy(\sigma)-\boldx(\sigma)}{|\boldy(\sigma)-\boldx(\sigma)|}=(-\sin \sigma,\cos\sigma).$$
It is trivial to observe that for $\sigma\in[0,\pi)$ one has $\Gamma_\sigma=\Gamma_{\sigma-\pi}$, $\boldx(\sigma)=\boldy(\sigma-\pi)$, and $\boldy(\sigma)=\boldx(\sigma-\pi)$. 
Using Lemma \ref{disk_cond},
for any $P\in \partial K$  the unit tangent vector to $\partial K$ in $P$ oriented in anti-clockwise way, $\boldtau(P)$, is a lipschitz function of $P$. Finally, we define the signed opening angle $\theta(\sigma)$ of the optimal arc having terminal points in $\boldx(\sigma)$ and $\boldy(\sigma)$ as the angle between  $\boldtau(\boldx(\sigma))$ and $-\boldtau(\boldy(\sigma))$, that is, the angle in the interval $[-\pi,\pi)$ such that an anti-clockwise $\theta(\sigma)$-rotation brings $\boldtau(\boldx(\sigma))$ into $-\boldtau(\boldy(\sigma))$. We refer to Figure \ref{CHL_fig} for notation.

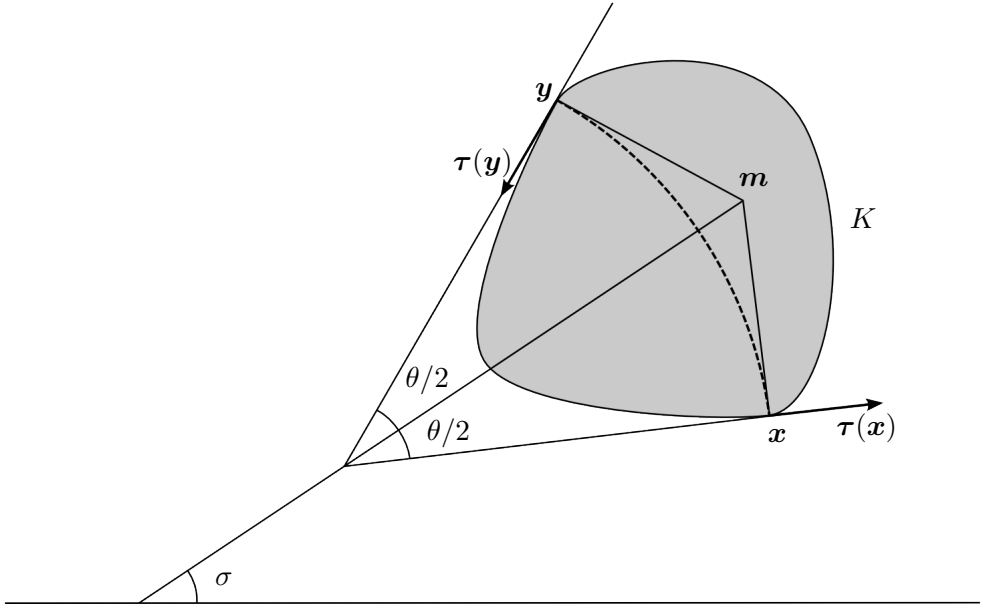
\begin{figure}
\includegraphics[height=80mm]{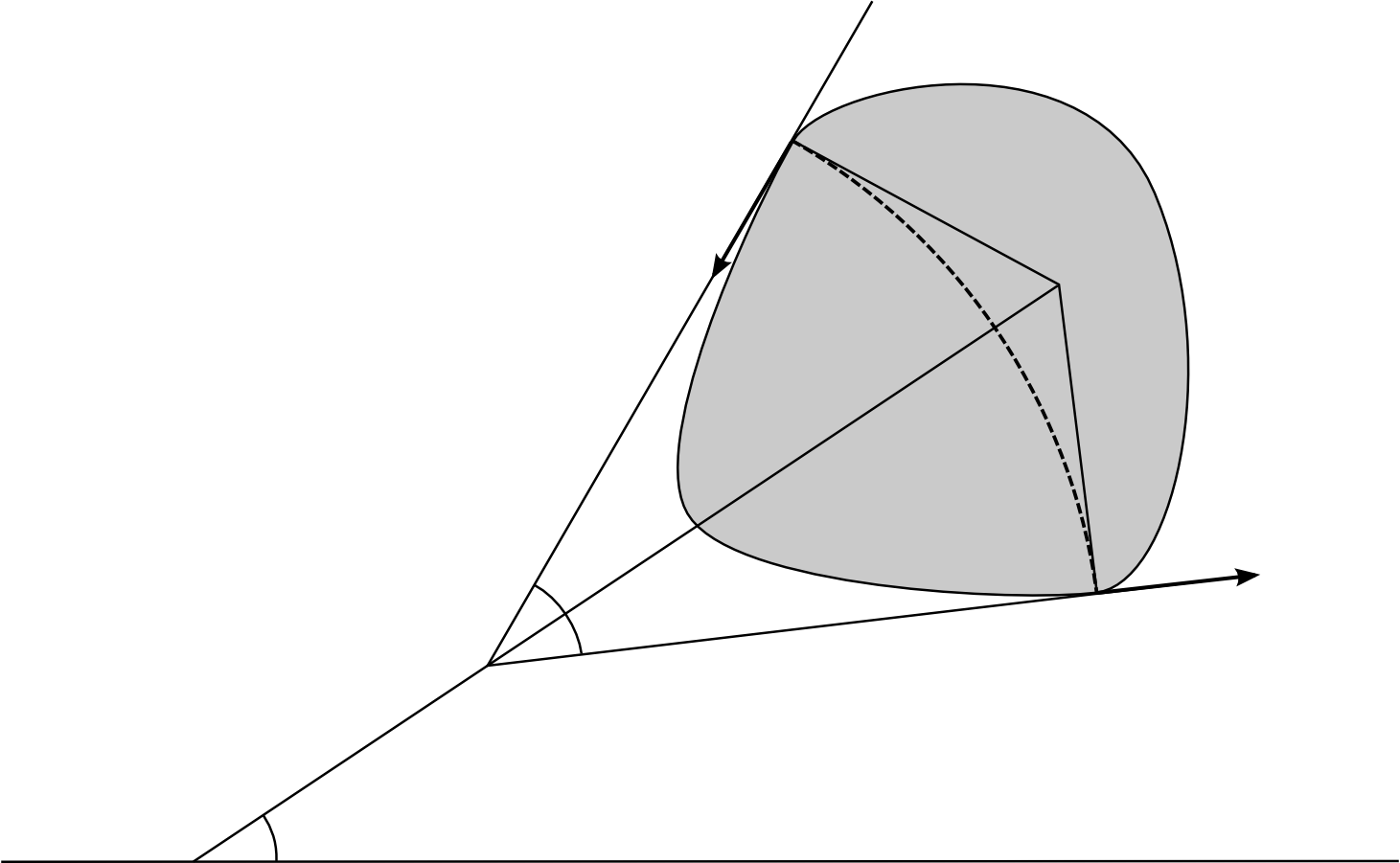}\caption{Notation for CHL-set}\label{CHL_fig}
\begin{picture}(0,0)
 \put(-33,120){$\theta/2$}
  \put(-25,100){$\theta/2$}
 \put(93,196){$\boldm$}
 \put(135,180){$K$}
\put(130,102){${\boldtau(\boldx)}$}
\put(104,99){${\boldx}$}
\put(-15,202){${\boldtau(\boldy)}$}
\put(16,230){${\boldy}$}
\put(-105,45){$\sigma$}
\end{picture}

\end{figure}

We prove the following lemma.
\begin{lemma}\label{regular_repr}
With the above notation, the function $\theta(\sigma)$ and the representations of $\partial K$, $\boldx(\sigma)$, $\boldy(\sigma)$, are of class $C^{0,1}$.
\end{lemma}
\begin{proof}
By definition, for $\sigma\in[-\pi,\pi)$, $\boldx(\sigma)$ and $\boldy(\sigma)$  are continuous anti-clockwise representations of $\partial K$.
Referring to Figure \ref{sin}, for $\varepsilon>0$ sufficiently small, 
we get
\begin{align}\label{sin2}
\sin \varepsilon =& \frac{ \sin \alpha_{\varepsilon} }{ |\boldx(\sigma) -\boldy(\sigma)|}(|\boldx(\sigma+\varepsilon) -\boldx(\sigma)|+|\boldy(\sigma+\varepsilon) -\boldy(\sigma)|)\\ \notag\\ \notag&+o(|\boldx(\sigma+\varepsilon) -\boldx(\sigma)| +|\boldy(\sigma+\varepsilon) -\boldy(\sigma)|),
\end{align}
where $\alpha_\varepsilon$ is the angle that the line  joining $\boldx(\sigma+\varepsilon)$ and $\boldy(\sigma+\varepsilon)$ forms with the two tangent lines to $\partial K$ at those points. 
Therefore, $\sin \alpha_\varepsilon= \cos(\theta(\sigma+\varepsilon)/2)$ and, dividing \eqref{sin2} by $\varepsilon$, after passing to the limit as $\varepsilon\to 0$, we have
\begin{align}\label{sin3}
&\limsup_{\varepsilon\to 0} \dfrac{|\boldx(\sigma+\varepsilon) -\boldx(\sigma)|}{\varepsilon}\le\sup_{s\in[-\pi,\pi[}\frac{ |\boldx(s) -\boldy(s)|}{\cos(\theta(s)/2)}\le \frac L{\cos(\sqrt3/2)},\\ \notag
\\ \label{sin4}
&\limsup_{\varepsilon\to 0} \dfrac{|\boldy(\sigma+\varepsilon) -\boldy(\sigma)|}{\varepsilon}\le\sup_{s\in[-\pi,\pi[}\frac{ |\boldx(s) -\boldy(s)|}{\cos(\theta(s)/2) }\le \frac L{\cos(\sqrt3/2)}.
\end{align}
We end the proof observing that $\theta(\sigma)$ is defined as the angle between two vectors which have a lipschitz dependence on $\sigma$.

\end{proof}
\setlength\fboxsep{30pt}
\setlength\fboxrule{0.0pt}
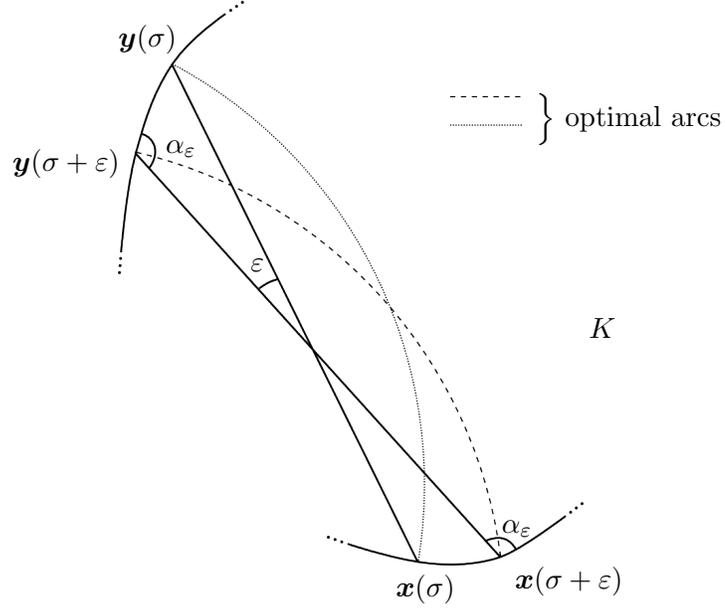
\begin{figure}
\fbox{\includegraphics[height=75mm]{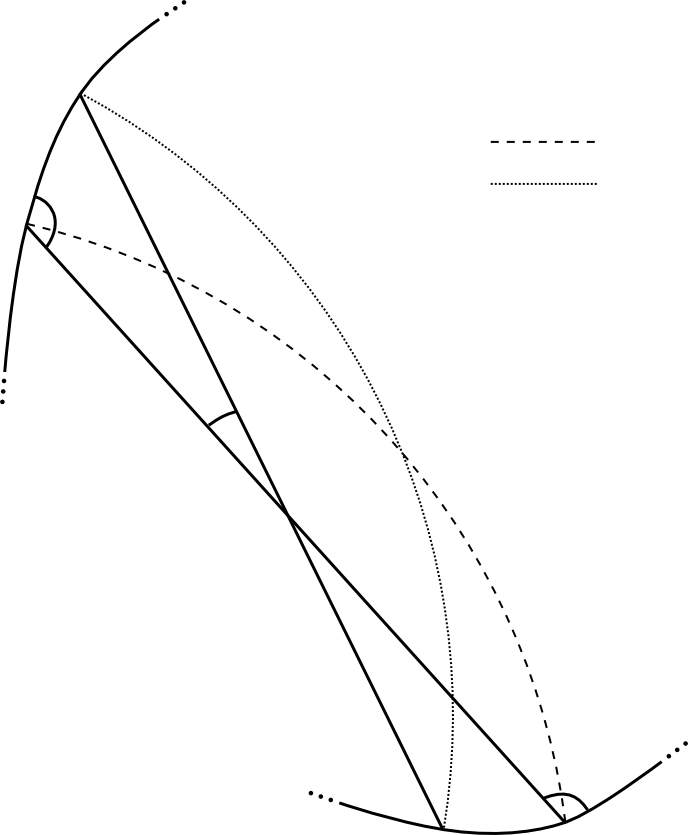}}\caption{Notation used in the proof of Lemma \ref{regular_repr}}\label{sin}
\begin{picture}(0,50)(0,0)
 \put(-38,215){$\varepsilon$}
 \put(90,190){$K$}
 \put(70,270){$\Big\}$ optimal arcs}
\put(-70,260){$\alpha_\varepsilon$}
\put(57,116){$\alpha_\varepsilon$}
\put(17,92){${\boldx}(\sigma)$}
\put(62,95){${\boldx}(\sigma+\eps)$}
\put(-128,253){${\boldy}(\sigma+\eps)$}
\put(-88,300){${\boldy}(\sigma)$}
\end{picture}
\end{figure}

We denote (see Figure \ref{CHL_fig}) by ${\boldm}(\sigma)$ the parametric representation of the intersection point of the straight lines tangent to $\Gamma_\sigma$ in ${\boldx(\sigma)}$ and ${\boldy}(\sigma)$ (if $\Gamma_\sigma$ is a straight segment ${\boldm}(\sigma)$ is its midpoint).
We have:
\begin{align}
\label{paramx}  \displaystyle{\boldx(\sigma)} &= {\boldm}(\sigma)-g(\theta(\sigma)) (-\sin(\sigma -\theta(\sigma)/2), \cos(\sigma -\theta(\sigma)/2)), \\\notag\\
\label{paramy} \displaystyle{\boldy(\sigma)} &= {\boldm}(\sigma)+g(\theta(\sigma)) (-\sin(\sigma +\theta(\sigma)/2), \cos(\sigma +\theta(\sigma)/2)),
\end{align}
where $g:]-\pi,\pi[\rightarrow\R$ is the smooth function defined by
\begin{equation}\label{fung}
g(\tau)=\left\{\begin{array}{ll}
\displaystyle\frac{L}{\tau}\tan\frac{\tau}{2}  &  \text{ if }\tau\not=0,    \\ \\
\displaystyle\frac{L}{2}  &  \text{ if }\tau=0.
\end{array}
\right.
\end{equation}

In view of Lemma \ref{regular_repr},  ${\boldx}(\sigma)$, ${\boldy}(\sigma)$,  $\theta(\sigma)$ and ${\boldm}(\sigma)$ are lipschitz functions. Therefore we may differentiate (\ref{paramx}) and (\ref{paramy}) a.e. and obtain
\begin{align}
\label{paramdx}  \displaystyle{\boldx'(\sigma)} &= {\boldm'}(\sigma)-\frac d{d\sigma}(g(\theta(\sigma))) (-\sin(\sigma -\theta(\sigma)/2), \cos(\sigma -\theta(\sigma)/2)) \\ \notag \displaystyle&\qquad+g(\theta(\sigma))\left(1-\frac{\theta'(\sigma)}{2}\right) (\cos(\sigma -\theta(\sigma)/2), \sin(\sigma -\theta(\sigma)/2)),
\end{align}
\begin{align}
\label{paramdy}  \displaystyle{\boldy'(\sigma)} &= {\boldm'}(\sigma)+\frac d{d\sigma}(g(\theta(\sigma))) (-\sin(\sigma +\theta(\sigma)/2), \cos(\sigma +\theta(\sigma)/2)) \\\notag& \displaystyle\qquad-g(\theta(\sigma))\left(1+\frac{\theta'(\sigma)}{2}\right) (\cos(\sigma +\theta(\sigma)/2), \sin(\sigma +\theta(\sigma)/2)).
\end{align}
The fact that the optimal arc touches the boundary of $K$ orthogonally implies the following conditions:
\begin{align}
\label{orthox} & \displaystyle{\boldx'(\sigma)} \cdot (-\sin(\sigma -\theta(\sigma)/2), \cos(\sigma -\theta(\sigma)/2)) =0, \\\notag\\
\label{orthoy} & \displaystyle{\boldy'(\sigma)} \cdot  (-\sin(\sigma +\theta(\sigma)/2), \cos(\sigma +\theta(\sigma)/2))=0.
\end{align}
Using \eqref{paramdx} and \eqref{paramdy}, from conditions \eqref{orthox} and \eqref{orthoy} we get
\begin{equation}
\label{orthom} {\boldm'(\sigma)} \cdot (-\sin \sigma, \cos \sigma) =0.
\end{equation}
This means that for a.e. $\sigma$ the vector  ${\boldm'(\sigma)}$ points in the direction of $(\cos \sigma, \sin \sigma)$, i.e. it can be written in the form
\begin{equation}
\label{mprime} {\boldm'(\sigma)}=M(\sigma) (\cos \sigma, \sin \sigma).
\end{equation}

Using again \eqref{paramdx} and \eqref{paramdy}, we can then write:
\begin{align*}
\notag  \displaystyle{\boldx'(\sigma)} &= \left[M(\sigma)\sin\frac{\theta(\sigma)}{2}-\frac d{d\sigma}(g(\theta(\sigma))) \right](-\sin(\sigma -\theta(\sigma)/2), \cos(\sigma -\theta(\sigma)/2)) \\\notag&  \displaystyle+\left[M(\sigma)\cos\frac{\theta(\sigma)}{2}+g(\theta(\sigma))\left(1-\frac{\theta'(\sigma)}{2}\right)\right] (\cos(\sigma -\theta(\sigma)/2), \sin(\sigma -\theta(\sigma)/2)) \\\notag\\
\notag \displaystyle{\boldy'(\sigma)} &= \left[-M(\sigma)\sin\frac{\theta(\sigma)}{2}+\frac d{d\sigma}(g(\theta(\sigma))) \right](-\sin(\sigma +\theta(\sigma)/2), \cos(\sigma +\theta(\sigma)/2)) \\\notag& \displaystyle+\left[M(\sigma)\cos\frac{\theta(\sigma)}{2}-g(\theta(\sigma))\left(1+\frac{\theta'(\sigma)}{2}\right)\right] (\cos(\sigma +\theta(\sigma)/2), \sin(\sigma +\theta(\sigma)/2)).
\end{align*}

As a consequence of \eqref{orthox} and \eqref{orthoy} we obtain the following condition which gives $M(\sigma)$ for a e. $\sigma$ in terms of $\theta(\sigma)$:
\begin{equation}\label{maincondm}
M(\sigma)=\frac{1}{\sin\frac{\theta(\sigma)}{2}}\frac d{d\sigma}(g(\theta(\sigma))),
\end{equation}
where, by continuity, $M(\sigma)=L\theta'(\sigma)/6$ whenever $\theta(\sigma)=0$.

We have now proved the following result.
\begin{proposition}\label{CHL-repr}
Let $K$ be a CHL-set. There exists a lipschitz function $\theta(\sigma)$, $\sigma\in[-\pi,\pi[$, which satisfies the following antisymmetry property
\begin{equation}\label{thetaper}
\theta(\sigma-\pi)=-\theta(\sigma),\qquad\forall\sigma\in[0,\pi[,
\end{equation}
such that,
using the notation of Figure \ref{CHL_fig}, the parametric representations ${\boldx(\sigma)}$ and ${\boldy(\sigma)}$ of the boundary of $K$ given in \eqref{paramx} and \eqref{paramy} are lipschitz and the following equalities hold true:
\begin{align}
\label{paramdm} \displaystyle{\boldm'(\sigma)}&= M(\sigma)(\cos \sigma, \sin \sigma)\\\notag\\
\label{paramdxx}  \displaystyle{\boldx'(\sigma)} &= \left[M(\sigma)\cos\frac{\theta(\sigma)}{2}+g(\theta(\sigma))\left(1-\frac{\theta'(\sigma)}{2}\right)\right] {\bf e}_-(\sigma) \\\notag\\
\label{paramdyy}  \displaystyle{\boldy'(\sigma)} &= \left[M(\sigma)\cos\frac{\theta(\sigma)}{2}-g(\theta(\sigma))\left(1+\frac{\theta'(\sigma)}{2}\right)\right] {\bf e}_+(\sigma),
\end{align}
where $M(\sigma)$ is given by \eqref{maincondm},  and where ${\bf e}_\pm(\sigma)$ denote the unit vectors
\begin{equation*}
{\bf e}_\pm(\sigma)=(\cos(\sigma \pm\theta(\sigma)/2), \sin(\sigma \pm\theta(\sigma)/2)).
\end{equation*}
\end{proposition}

Using the above notation, from \eqref{paramx}, \eqref{paramy} and  \eqref{paramdm}  we have
\begin{align}
\label{paramxx}  \displaystyle{\boldx(\sigma)} &= \boldm(0)+\int_0^\sigma M(s)(\cos s, \sin s)\,ds-g(\theta(\sigma)) (-\sin(\sigma -\theta(\sigma)/2),\cos(\sigma -\theta(\sigma)/2)), \\\notag\\
\label{paramyy} \displaystyle{\boldy(\sigma)} &=\boldm(0)+ \int_0^\sigma M(s)(\cos s, \sin s)\,ds+g(\theta(\sigma)) (-\sin(\sigma +\theta(\sigma)/2), \cos(\sigma +\theta(\sigma)/2)).
\end{align}
By the Gauss-Green formula, it is possible to compute the measure of $K$:
\begin{align}
\label{area}
|K|=&\frac12\int_0^\pi(\boldx(\sigma)\wedge\boldx'(\sigma)+\boldy(\sigma)\wedge\boldy'(\sigma))\,d\sigma\\ \notag\\ \notag
=&\int_0^\pi\int_0^tM(t)M(\sigma)\,\sin(t-\sigma)\,d\sigma\,dt+\int_0^\pi g^2(\theta(\sigma))\,d\sigma.
\end{align}
We observe explicitly: If for some lipschitz function $\theta:[-\pi,\pi[\rightarrow[-\pi,\pi[$ satisfying \eqref{thetaper} the equations \eqref{paramxx} and \eqref{paramyy}
define the boundary of a convex set $K$, then $K$ is a CHL-set.
Indeed, by construction for every $\sigma\in[-\pi,\pi[$ there is a  circular arc $\omega_\sigma$ of length $L$ and opening angle $\theta(\sigma)$, which is orthogonal to $\partial K$ and has terminal points in $\boldx(\sigma)$ and $\boldy(\sigma)$. For every $\sigma\in[-\pi,0[$ this arc bounds two parts of $K$ denoted as  $K^{(1)}_\sigma$ and $K^{(2)}_\sigma$, whose measures are given by
\begin{align}
\label{K1}
|K^{(1)}_\sigma|=&\frac12\int_\sigma^{\pi+ \sigma}(\boldx(\tau)-\boldm(\sigma))\wedge\boldx'(\tau)\,d \tau+\frac{L^2}{2\theta(\sigma)}\left(\frac2{L}g(\theta(\sigma))-1\right)\\ \notag\\\label{K2}
|K^{(2)}_\sigma|=&\frac12\int_ \sigma^{\pi+ \sigma}(\boldy(\tau)-\boldm(\sigma))\wedge\boldy'(\tau)\,d \tau-\frac{L^2}{2\theta(\sigma)}\left(\frac2{L}g(\theta(\sigma))-1\right).
\end{align}
Here we have used the Gauss-Green formula and the observation that the measure of the set bounded by $\omega_\sigma$ and the two segments joining $\boldm(\sigma)$ to $\boldx(\sigma)$ and to $\boldy(\sigma)$ is given by the modulus of the second term on the right-hand side of \eqref{K1} and \eqref{K2}. A straightforward calculation gives
\begin{align}
\notag
|K^{(1)}_\sigma|=&\frac12\int\limits_\sigma^{\pi+\sigma}\int\limits_0^tM(t)M(\tau)\,\sin(t-\tau)\,d \tau\,dt+\frac12\int\limits_\sigma^{\pi+\sigma} g^2(\theta(\tau))\,d \tau\\ \notag\\ \notag
&
+\frac12\int\limits_\sigma^{\pi+ \sigma}(\boldm(0)-\boldm(\sigma))\wedge\boldx'(\tau)\,d \tau-\frac12\left.\left(g(\theta(\tau))\int\limits_0^\tau M(t)\cos(\tau-t-\theta(\tau)/2)\,dt\right)\right|_{\tau=\sigma}^{\tau=\pi+\sigma}\\ \notag\\ \notag
&+\int\limits_\sigma^{\pi+\sigma}\left(M(\tau)g(\theta(\tau))\cos\left(\frac{\theta(\tau)}2\right)-g^2(\theta(\tau))\frac{\theta'(\tau)}4\right)\,d\tau+\frac{L^2}{2\theta(\sigma)}\left(\frac2{L}g(\theta(\sigma))-1\right)\\ \notag\\ \notag
=&I_1+I_2+I_3+I_4+I_5+I_6.
\end{align}
In view of property \eqref{thetaper} we have that
\begin{equation*}
I_1+I_2=\frac{|K|}2.
\end{equation*}
Furthermore, using \eqref{paramdm}, we have
\begin{align}\notag
I_4=&-\frac12g(\theta(\sigma))\cos\left(\frac{\theta(\sigma)}2\right)\left[\int_\sigma^0M(\tau)\cos(\sigma-\tau)\,d\tau-\int_0^{\pi+\sigma}M(\tau)\cos(\sigma-\tau)\,d\tau\right]\\ \notag\\ \notag
&=-g(\theta(\sigma))\cos\left(\frac{\theta(\sigma)}2\right)(\boldm(0)-\boldm(\sigma))\wedge (-\sin\sigma,\cos \sigma)=-I_3.
\end{align}
Finally, it is easy to show that $I_5+I_6=0$ and it results
\begin{equation*}
|K^{(1)}_\sigma|=|K^{(2)}_\sigma|=\frac{|K|}2.
\end{equation*}

\begin{remark}\label{fig9}  A particular example of a CHL-set can be obtained if we choose
\begin{equation*}
\theta(\sigma)=\frac{\pi-\bigl|2\pi-|6\sigma-3\pi| \bigr|}3,\qquad \sigma\in[0,\pi[,
\end{equation*}
in the representation \eqref{paramx}, \eqref{paramy}. This set is a kind of rounded equilateral triangle (see Figure \ref{auerb}), which can be seen as the counterpart of the Auerbach triangle in the class of convex Zindler sets. Just for comparison, its measure is larger than the measure of the disc with the same value of $L$; using \eqref{area} its area is given by
\begin{equation*}
L^2\left(\frac9\pi-2\sqrt3\left(\frac3{2\pi}+I\right)^2\right)\simeq L^2\cdot0.7981\dots,
\end{equation*}
where
\begin{equation*}
I=\int_0^{\pi/3}\frac{\cos t}t\left(\frac1t-\frac1{\tan t}\right)\,dt\simeq 0.2949\dots
\end{equation*}
Similarly, from any regular polygon with an odd number of sides one can obtain a CHL-set. The above example also shows that in general a CHL-set is not necessarily more regular than of class $C^{1,1}$ and even though it  is of class $C^{1,1}$, the function $\theta(\sigma)$ does not need to be more regular than of class $C^{0,1}$.
\begin{figure}
\includegraphics[width=3. in]{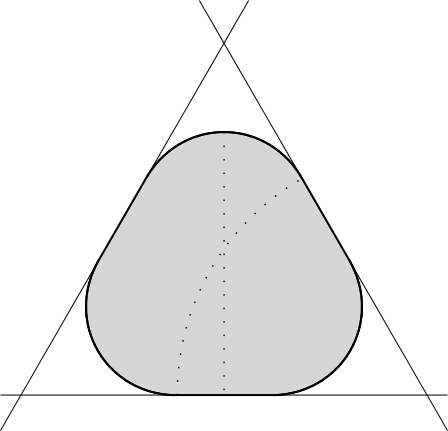}\caption{The set from Remark \ref{fig9}}\label{auerb}
\end{figure}
\end{remark}

\subsection{Maximization of \eqref{const} on CHL-sets}

In the present subsection we maximize the quantity in \eqref{const}
in the class of CHL-sets. 
\begin{proposition}\label{miniimality}
Let $K$ be a CHL-set such that the length of the optimal arcs is $L$. The measure of $K$ is minimal if and only if it is a disc of diameter $L$.
\end{proposition}
\begin{proof}
We have to prove that for a given $L$ the integral in \eqref{area} attains its minimum if and only if $\theta(\sigma)=0$, for every $\sigma\in[-\pi,\pi[$. In that case $M(\sigma)\equiv 0$ and
\begin{equation}\label{minimal}
|K|=\frac\pi4 L^2.
\end{equation}
From \eqref{mprime} we immediately derive
\begin{equation}\label{orthoMM}
\int_0^\pi M(\sigma)\sin \sigma\,d \sigma =\int_0^\pi M(\sigma)\cos \sigma\,d \sigma =0.
\end{equation}
Set
\begin{equation}\label{funf}
f(\tau)=\frac{L}2\int_0^\tau\frac1{t\sin (t/2)} \left(1-\frac2t \tan\frac t2+\tan^2\frac{t}{2}\right)\,dt,
\end{equation}
and observe that by \eqref{maincondm} and \eqref{fung}
\begin{equation}
\label{funf2}
\frac d{dt}f(\theta(t))=M(t).
\end{equation}
As a consequence of \eqref{thetaper} and \eqref{orthoMM}, $f$ satisfies the following properties:
\begin{align}
\label{propf1}&f(\theta(\sigma))=-f(\theta(\sigma-\pi)),\qquad \sigma\in[0,\pi[,\\ \notag\\ 
\label{propf2}&\int_0^\pi f(\theta(\sigma))\sin\sigma\,d\sigma=\int_0^\pi f(\theta(\sigma))\cos\sigma\,d\sigma=0,\\\notag\\
\label{propf3}&\int_{-\pi}^\pi f(\theta(\sigma))\sin n\sigma\,d\sigma=\int_{-\pi}^\pi f(\theta(\sigma))\cos n\sigma\,d\sigma=0,\quad \text{ for $ n=0$ and $n$ even}.
\end{align}
By \eqref{funf2} and \eqref{propf2}, 
a double integration by parts allows us to write \eqref{area} as follows
\begin{equation}\label{areaequiv}
|K|=\int_0^\pi\int_0^tf(\theta(t))f(\theta(\sigma))\,\sin(t-\sigma)\,d\sigma\,dt-\int_0^\pi f^2(\theta(\sigma))\,d\sigma+\int_0^\pi g^2(\theta(\sigma))\,d\sigma.
\end{equation}
Using the antisymmetry property \eqref{thetaper} of $\theta(\sigma)$, one can prove that
\begin{equation}\label{inequa}
\int_0^\pi\int_0^tf(\theta(t))f(\theta(\sigma))\,\sin(t-\sigma)\,d\sigma\,dt+\frac18\int_0^\pi f^2(\theta(\sigma))\,d\sigma\ge0.
\end{equation}
Indeed, in view of \eqref{propf1}, \eqref{propf2} and \eqref{propf3} $f$ can be written as a Fourier series in $[-\pi,\pi[$
\begin{equation}\label{series}
f(\theta(\sigma))=\sum_{n=1}^{+\infty}a_n\cos (2n+1)\sigma+\sum_{n=1}^{+\infty}b_n\sin (2n+1)\sigma,
\end{equation}
so that a direct computation gives
\begin{align}\label{series1}
\int_0^tf(\theta(\sigma))\,\sin(t-\sigma)\,d\sigma=&-
\sum_{n=1}^{+\infty}\frac{a_n\cos (2n+1)t+b_n\sin (2n+1)t}{(2n+1)^2-1}\\ \notag\\
\notag&+\sum_{n=1}^{+\infty}\frac{a_n\cos t}{(2n+1)^2-1}+\sum_{n=1}^{+\infty}\frac{b_n(2n+1)\sin t}{(2n+1)^2-1}.
\end{align}
Computing the first integral in \eqref{inequa}, we have
\begin{align}\label{series2}
\int_0^\pi\int_0^tf(\theta(t))f(\theta(\sigma))\,\sin(t-\sigma)\,d\sigma\,dt&=-\frac\pi2
\sum_{n=1}^{+\infty}\frac{a_n^2+b_n^2}{(2n+1)^2-1}
\\ \notag\\
\notag&\ge-\frac\pi{16}
\sum_{n=1}^{+\infty}(a_n^2+b_n^2)=-\frac1{8}\int_0^\pi f^2(\theta(\sigma))\,d\sigma.
\end{align}
This means that \eqref{inequa} is completely proved, and from \eqref{areaequiv} we get now
\begin{equation}\label{areaequiva}
|K|\ge-\frac98 \int_0^\pi f^2(\theta(\sigma))\,d\sigma+\int_0^\pi g^2(\theta(\sigma))\,d\sigma
\end{equation}

It is not difficult to show that the functions $f$ and $g$ given in \eqref{funf} and \eqref{fung} satisfy the following inequality:
\begin{equation}\label{ineqfg}
g^2(\tau)-\frac98 f^2(\tau)\ge\frac{L^2}4, \qquad \tau\in[-\sqrt3,\sqrt3],
\end{equation}
where equality holds if and only if $\tau=0$.
Indeed, for $0\le\tau<\pi/2$, inequality \eqref{ineqfg} is an immediate consequence of the following one:
\begin{equation}\label{ineq}
\sin^2(\tau/2)+\frac19 \frac{\sin^4(\tau/2)}{\cos^2(\tau/2)}\le\left(\frac\tau2\right)^2.
\end{equation}
A simple calculation proves that \eqref{ineq} holds true for $\tau\in]0,\pi/2[$. The convexity of $f$ and $g$ in $[\pi/2,\sqrt3]$ yields inequality \eqref{ineqfg} also in this interval. A similar computation can be carried out when $\tau<0$ and then \eqref{ineqfg} is completely proved. 

Using inequality \eqref{ineqfg} and Proposition \ref{cianchi}\emph{(d)}  we conclude that
\begin{equation*}
|K|\ge\frac\pi4 L^2,
\end{equation*}
and equality holds if and only if $K$ is a disc.

\end{proof}

\section{Proof of Theorem \ref{mainauerb}}\label{sect_main10}

The proof of Theorem \ref{mainauerb} will be achieved following arguments which are very similar to those used in the proof of Theorem \ref{main}. We start with some definitions and preliminary results.

Let $K$ be an open convex set of $\R^2$. We set
\begin{equation}
\label{gonst}
\mathcal{G}(K) =  \mathop{\inf_{F\subset\R^2\text{ half-plane}}}_{0<|F\cap K| \le \frac{|K|}{2}} \mathcal{T}(F;K)
\end{equation}
where
\begin{equation}
\label{quotientc}
\mathcal{T}(F;K) =  \dfrac{Per(F\cap K;K)^2}{|F\cap K|}.
\end{equation}
In what follows we say that a half-plane $H$ is a minimizer for  \eqref{quotientc} if the minimum in \eqref{gonst} is attained 
on $H$. 

\begin{proposition}
\label{cianchic}
Let $K$ be an open convex set of $\R^2$. 
 There exists a half-plane $H$ which minimizes \eqref{quotientc} such that  $|H\cap K|=|K|/2$, and any minimizer $H$ has the following properties.
\begin{enumerate}[(a)]
\item Let $P$ be one of the terminal points of $\partial H\cap K $. Then $P$ is a regular point of $\partial K$ in the sense that $\partial K$ has a tangent at $P$.
Furthermore, the tangents to $\partial K$ at the two terminal points of $\partial H\cap K $ either bound  with $\partial H  \cap  K$ an isosceles triangle or they are orthogonal to  $\partial H  \cap  K$.
\item If $|H  \cap  K|<|K|/2$, then $H  \cap  K$ is an isosceles triangle having sides on $\partial K$.
In such a case there exists another minimizer $\hat H$ such that  $\hat H  \cap  K$ is an isosceles triangle with sides on $\partial K$, having the same vertex as $H  \cap  K$, and $|\hat H  \cap  K|=|K|/2$.
\end{enumerate}
\end{proposition}
\begin{proof} We divide the proof into three parts.

\setcounter{claim}{0}
\begin{claim}\label{chord1} The minimum in  \eqref{gonst} is achieved and there exists a minimizer $H$ such that $|H\cap K|=\dfrac{|K|}{2}$.
\end{claim}

\noindent Given a half-plane $H$, we denote by $H_\delta$ the half-plane such that $H\subset H_\delta$ and the strip $H_\delta\backslash H$ has width $\delta$. We prove that for any $H$ such that $0<|H\cap K|<|K|/2$ there exists $\delta>0$ such that 
\begin{align}\label{Hdelta}
&0<|H_\delta\cap K|\le|K|/2,\\\notag\\
\label{Hdelta1}&\mathcal{T}(H_\delta;K)\le \mathcal{T}(H;K).
\end{align}
Consider $\delta>0$ sufficiently small, such that \eqref{Hdelta} is verified. Two cases may occur:
\begin{enumerate}[a)]
\item at the terminal points of $\partial H\cap K$ one can find two supporting lines for $K$ (one for each terminal) which are parallel or meet in a point which is external to $H$;
\item the above condition is not verified.
\end{enumerate}

In case a) it is immediate to observe that $Per(H_\delta
\cap K; K)\le Per(H
\cap K; K)$, while $|H_\delta
\cap K|>|H
\cap K|$. Then  \eqref{Hdelta1} is verified as a strict inequality.

In case b) we have that there exist two supporting lines for $K$ at the terminal points of $\partial H\cap K$ which meet in a point $P$ which is internal to $H$ and we denote by $T$ the triangle bounded by such lines and $\partial H\cap K$. 
A straightforward computation gives
\begin{equation}
\label{THK}
\mathcal{T}(H_\delta;K) \le
\mathcal{T}(H;K) \left( 1+\delta Per(H\cap K;K)\left(\frac1{|T|}-\frac1{|H\cap K|}\right)+o(\delta)\right).
\end{equation}
This means that, if $H\cap K\not=T$, then for sufficiently small $\delta>0$, \eqref{Hdelta1} is satisfied. On the other hand, if $H\cap K=T$ and $H_{\delta}\cap K$ is a triangle similar to $T$, then  \eqref{Hdelta1} immediately holds, while  if $H\cap K=T$ and $H_{\delta}\cap K$ is not a triangle, then
\begin{equation}
\label{THKa}
\mathcal{T}(H_\delta;K) < \frac{2Per(H_\delta\cap K;K)}{\dist(P,\partial H_\delta)}
\le \frac{2Per(H\cap K;K)}{\dist(P,\partial H)}=
\mathcal{T}(H;K),
\end{equation}
that is, \eqref{Hdelta1} is verified also in this case.

The result just proved says that a minimizing sequence $\{H_n\}_{n\in\N}$ for the minimum problem \eqref{gonst} can be chosen in such a way that $|H_n|=|K|/2$. From usual compactness arguments  we get Claim \ref{chord1}.

\begin{claim}\label{chord2} Property (a) holds true.
\end{claim}

\noindent Let us consider a coordinate system $(\xi,\eta)$ in $\R^2$ such that $\partial H\cap K $ has length $L$ and its terminal points in $P_1=(0,-L/2)$ and $P_2=(0,L/2)$.
In a neighborhood of $P_1$ the boundary of $K$ can be described as the graph of a convex function $f(\xi)$, while in a neighborhood of $P_2$ the boundary of $K$ can be described as the graph of a concave function $g(\xi)$. Because of the convexity of $K$ there exist four real numbers $\alpha^\pm$, $\beta^\pm$ such that: 
\begin{align}\label{fgp}
f(\xi)=-\frac L2+\alpha^+\xi+o(\xi),\qquad g(\xi)=\frac L2+\beta^+\xi+o(\xi), \qquad\xi>0,\\ \notag\\
\label{fgm}
f(\xi)=-\frac L2+\alpha^-\xi+o(\xi),\qquad g(\xi)=\frac L2+\beta^-\xi+o(\xi), \qquad\xi<0,
\end{align}
where
\begin{equation}
\label{convexc}
\alpha^+\ge\alpha^-,\qquad \beta^+\le\beta^-.
\end{equation}

For $\epsilon>0$ we denote by $H^{\pm\epsilon}$ a half-plane bounded by the straight line $\eta=\pm\xi/\epsilon$ such that $|H^{\pm\epsilon}\cap K|\le |K|/2$. We have:
\begin{equation}\label{measc}
|H^{\pm\epsilon}\cap K|=\frac{|K|}2+o(\epsilon),
\end{equation}
while, for $\epsilon$ small enough, 
\begin{equation}\label{perc}
Per(H^{\pm\epsilon}\cap K;K)-Per(H\cap K;K)=\pm\epsilon\frac L2
(\alpha^\mp+\beta^\pm)+o(\epsilon).
\end{equation}
The optimality of $H$, \eqref{measc} and \eqref{perc} imply
\begin{equation}\label{optimc}
\alpha^++\beta^-\le0\le\alpha^-+\beta^+.
\end{equation}
Using \eqref{convexc} and \eqref{optimc} we have
\begin{equation*}
\alpha^+=\alpha^-=-\beta^+=-\beta^-,
\end{equation*}
which proves Claim \ref{chord2}.



\begin{claim}\label{chord4} Property (b) holds true.
\end{claim}

\noindent By the proof of Claim \ref{chord1}, it follows that $H  \cap  K$ and $H_\delta  \cap  K$, for $\delta>0$ small enough, have to be similar triangles having sides on $\partial K$ and sharing a vertex. Therefore there exists $\bar \delta>0$ such that $H_{\bar\delta}  \cap  K$ has measure $|K|/2$ and it is still a triangle similar to $H  \cap  K$ having sides on $\partial K$ and sharing a vertex. The fact that such triangles are isosceles comes from Claim \ref{chord2} and the proof of Claim \ref{chord4} is complete.

\end{proof}

\begin{definition}[Optimal chord]
If $H$ is a minimizer of \eqref{quotientc} such that $| H  \cap  K|=|K|/2$, then we say that $\partial H  \cap  K$ is an optimal chord of $K$.
\end{definition}

As for the proof of Theorem \ref{main}, in view of Proposition \ref{cianchic}, Theorem \ref{mainauerb} is a consequence of the following result.
\begin{theorem}
\label{main1auerb} If $K$ is an open convex set of $\R^2$, we have:
\begin{equation}\label{mainineqac}
\mathcal{G}(K) \le \mathcal{G}(K^{\triangle})=\frac{16}{\sqrt3(8\log3-\log^23-4)} \simeq 2.5789\dots,
\end{equation}
where $K^{\triangle}$ is the Auerbach triangle with the same measure as $K$. The inequality holds as an equality if and only if $K$ is an Auerbach triangle.
\end{theorem}

Before giving the proof of the above theorem it could be useful to recall the definition of the so-called Auerbach triangle (see, e.g., \cite[Appendix C]{DEGKR}). We consider the Auerbach triangle $A^\triangle $ (see Figure \ref{auerbach}) with length of the halving chords equal to 1. The boundary $\Gamma^ \triangle$ of $A^\triangle$ is $C^{1,1}$ and consists of six parts. We start by giving the parametric representation of one of these parts which in Figure \ref{auerbach} is represented by a dashed line:
\begin{equation}\label{param_auer}
\left\{
\begin{array}{l}
\displaystyle x(t)=\frac{e^{4t}-1}{e^{4t}+1}-t\\
\\
\displaystyle y(t)=2\frac{e^{2t}}{e^{4t}+1}
\end{array}
\right.
\end{equation}
with $t\in[-(\log3)/4,(\log3)/4]$. This arc of curve is clearly symmetric with respect to the $y$-axis. Moreover we have $(x(0),y(0))=(0,1)$, while at the terminal points of \eqref{param_auer} the tangent lines to the curve form an angle of $\pi/3$ . So we can construct an equilateral triangle $T$ bounded by these two tangents and the $x$-axis. 
Rotations of $2\pi/3$ around the barycenter of $T$ of the arc defined by \eqref{param_auer} provide other two pieces of boundary. The Auerbach triangle is just the convex hull of these three pieces. Clearly $A^ \triangle $ shares the flat part of its boundary with the equilateral triangle $T$, therefore sometimes goes under the name of \emph{Rounded Triangle}.  

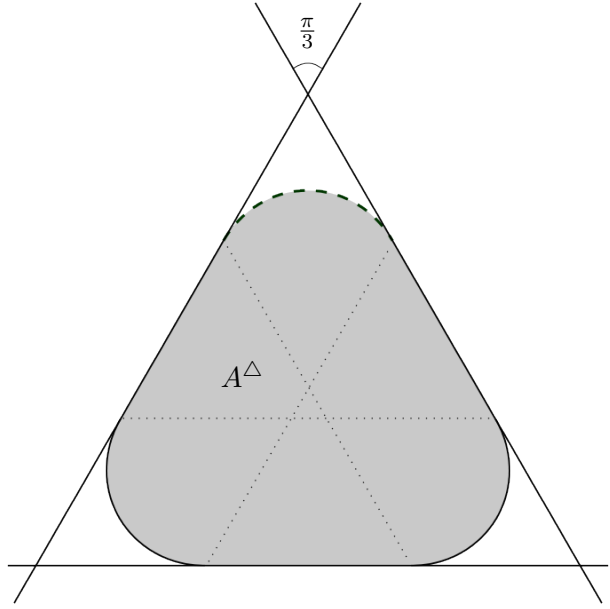
\begin{figure}
\includegraphics[height=80mm]{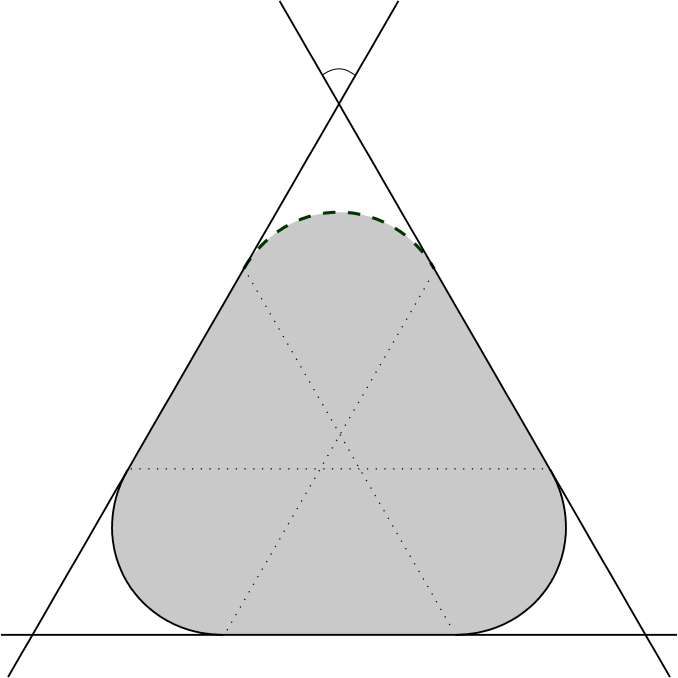}\caption{Auerbach triangle}\label{auerbach}
\begin{picture}(0,0)
\put(-33,120){$A^ \triangle$}
\put(-4,252){${\frac{\pi}{3}}$}
\end{picture}
\end{figure}

Using the described representation of $A^ \triangle $, it is possible to work out its measure
\begin{equation}\label{area_auerb}
|A^ \triangle |=\frac{\sqrt3}8(8\log3-\log^23-4)\simeq 0.7755\dots,
\end{equation}
from which the computation of $\mathcal{G}(K^{\triangle})$ in \eqref{mainineqac} easily follows.

\begin{remark}\label{remark_auerb}
As already said in the introduction, it is not difficult to show that the quantity $\mathcal{C}(A^{\triangle})$ is strictly less then $8/\pi$, even though $\mathcal{G}(A^{\triangle})$ is not. Indeed, in view of Proposition \ref{cianchi}, the optimal arc is not a straight segment and it connects the parts of $\Gamma^\triangle$ which are segments. A straightforward computation which makes use of \eqref{area_auerb} gives:
\begin{equation*}
\mathcal{C}(A^{\triangle})=\frac{8\pi}{3(8\log3-\log^23-4)} \simeq 2.3388\dots<\frac{\pi}{8}\simeq 2.5464\ldots\ .
\end{equation*}
\end{remark}

\begin{proof}[Proof of Theorem \ref{main1auerb}]
First of all  we observe that, as in Lemma \ref{max}, one can easily show that there exists a set $K^*$ such that
\begin{equation}\label{existencec}
\mathcal{G}(K^*)= \sup_K \mathcal{G}(K).
\end{equation}
We just have to prove that such a set $K^*$ belongs to the class of convex Zindler sets 
since inequality \eqref{mainineqac} holds true in such a class (see \cite{FP}).
The strategy of the proof is similar to the one outlined in Subsection \ref{sub_red} therefore we point out just few differences. Arguing with chords instead of arcs it is immediate to observe that two optimal chords cross each other in one point or they have to coincide. Corollary \ref{coroll} follows at once replacing circular sector with isosceles triangles and the counterpart of Proposition \ref{CHLlemma} consists in showing that if by contradiction $K^*$ is not a convex Zindler set then it is possible to cut off a piece of $K^*$
in such a way that $\mathcal G(\cdot)$ strictly increases.

\end{proof}

\section {Corollaries and remarks}\label{last}

In this section we collect a few more results and remarks. The first result addresses the question of best fencing for sets which are not necessarily convex. As announced in the introduction, we are able to handle general centrosymmetric sets.

\begin{proposition}\label{centrosym}
Let $K\subset\R^2$ be a measurable centrosymmetric set. Then 
\begin{equation}\label{maincentro}
\mathop{\inf_{G \subset K}}_{|G| =\frac{|K|}{2}}Per(G;K)^2\le\frac4\pi|K|.
\end{equation}
Moreover, equality holds in (\ref{maincentro}) if and only if $K$ is a disc.
\end{proposition}
\begin{proof}
Without loss of generality we can assume that the set $K$ is centrosymmetric about 0. 
Its measure is given by
\begin{equation*}
|K|=\int_0^{2\pi}\int_0^{\infty}\chi_{_K}(r,\theta)r\,dr\,d\theta,
\end{equation*}
where $\chi_{_K}(r,\theta)$ denotes the characteristic function of $K$ in polar coordinates. Using Hardy-Littlewood inequality,
we have:
\begin{align*}
|K|&\ge\int_0^{2\pi}\int_0^{\infty}\chi^*_{_K}(r,\theta)r\,dr\,d\theta
\\ \\
&=\frac12\int_0^{2\pi}\left(\int_0^{\infty}\chi_{_K}(r,\theta)\,dr\right)^2d\theta,
\end{align*}
where $\chi^*_{_K}(r,\theta)$ denotes the decreasing rearrangement of $\chi_{_K}(r,\theta)$ with respect to the variable $r$ (see, e.g., \cite{Ka}).

It follows that there exists $\hat\theta$ such that
\begin{equation*}
\int_0^{\infty}\chi_{_K}(r, \hat\theta)\,dr\le\sqrt{\frac{|K|}\pi}.
\end{equation*}
Since $K$ is centrosymmetric, if $\widehat H$ is the half-plane bounded by a straight line passing through the origin which forms an angle $\hat \theta$ with the positive $x$-axis, we have
\begin{equation*}
Per(\widehat H;K)\le \mathcal{H}^1(\partial(\widehat H\cap K))=2\int_0^{\infty}\chi_{_K}(r,\hat\theta)\,dr\le2\sqrt{\frac{|K|}\pi},
\end{equation*}
which implies \eqref{maincentro}.

As regards the equality case in \eqref{maincentro} we simply observe that all the inequalities stated above have to hold as equalities for every $\hat\theta\in[0,2\pi]$. Via standard arguments about the equality case in Hardy-Littlewood inequality (see, e.g., \cite[Appendix C]{ALT}) we conclude the proof.

\end{proof}

We briefly discuss Corollaries \ref{corol_intr} and \ref{corol_intr1}, giving a sketch of their proofs.

\begin{proof}[Proof of Corollary \ref{corol_intr}]
Theorem  \ref{main} gives the inequality for $\alpha =1/2$. For  $\alpha>1/2$ one can use the proof of Lemma 3.1 in \cite{C} to show that
\begin{equation*}
\gamma_\alpha(K) =\gamma_{1/2}(K)\left(\frac2{|K|}\right)^{\alpha-\frac12}.
\end{equation*}

\end{proof}

\begin{proof}[Proof of Corollary \ref{corol_intr1}]
In view of the results in \cite{Ga}, \cite{GG}, Corollary \ref{corol_intr} implies \eqref{32} and \eqref{33}.
As regards \eqref{34} we observe that
\begin{equation*}
I(K)\le |K|^{1/2}\gamma_{1}(K)=\sqrt2\gamma_{1/2}(K),
\end{equation*}
and that (see \cite{C1})
\begin{equation*}
I(K^\sharp)=\sqrt2\gamma_{1/2}(K^\sharp).
\end{equation*}

\end{proof}

\begin{remark}
The Szeg\"o-Weinberger-inequality states that among domains with fixed measure the first non trivial eigenvalue 
of the (linear) Laplacian operator under Neumann boundary conditions becomes maximal for balls. Therefore inequality \eqref{33} can be interpreted as
an extension of this result to the 1-Laplacian operator in convex planar domains.
\end{remark}

\begin{remark}
In \cite{BW} planes which cut centrosymmetric $n$--dimensional bodies into two halves of equal volume are considered. If $A(\Omega)$ denotes a cut through $\Omega$ which minimizes ($n-1$)--dimensional area, it is shown that $A(\Omega)\leq A(\Omega^\sharp)$. Another generalization of P\'olya's result to higher dimensions is described in \cite{KK}. Given the volume of a centrosymmetric set $\Omega$, only a ball maximizes the length of the shortest line segments running through the center of $\Omega$. In this sense the ball has the longest shortest piercing.
\end{remark}

\vskip1.2cm
\small
\noindent
Luca Esposito, Dipartimento di Matematica e Informatica, Via Ponte Don 
Melillo, I-84084 Fisciano (SA), Italy.

\noindent email: luesposi@unisa.it

\bigskip\noindent
Vincenzo Ferone, Carlo Nitsch, Cristina Trombetti,
Dipartimento di Matematica e Applicazioni
\lq\lq R. Caccioppoli\rq\rq$ $, 
Universit\`a di Napoli \lq\lq
Federico II\rq\rq $ $, 
Complesso Universitario Monte S. Angelo, 
Via Cintia, 
I-80126 Napoli, 
Italy.

\noindent email: ferone@unina.it, c.nitsch@unina.it, cristina@unina.it

\bigskip\noindent
Bernd Kawohl, 
Mathematisches Institut, 
Universit\"at zu K\"oln, 
D-50923 K\"oln, 
Germany.

\noindent email: kawohl@mi.uni-koeln.de

\end{document}